\documentclass{article}

\usepackage{authblk}

\usepackage[english]{babel}
\usepackage[utf8]{inputenc}
\usepackage{amsmath,amsthm}
\usepackage{graphicx}
\usepackage{caption}
\usepackage{subcaption}
\usepackage[colorinlistoftodos]{todonotes}
\usepackage{amssymb}
\usepackage{xcolor}
\usepackage{fullpage}

\usepackage{mathtools} 
\usepackage{xargs} 

\usepackage[inline]{enumitem} 

\numberwithin{equation}{section} 

\usepackage{setspace} 

\usepackage{nameref}

\usepackage{needspace}

\definecolor{lgreen}{rgb}{0.0, 0.48, 0.0}
\definecolor{lpurple}{rgb}{0.48, 0.0, 0.48}
\usepackage[breaklinks]{hyperref}
\usepackage[
    capitalize,
    nameinlink,
    noabbrev
]{cleveref}
\usepackage{url}
\hypersetup{linktocpage,
            colorlinks=true,
            linkcolor=lgreen,
            citecolor=lpurple}

\definecolor{bblue}{rgb}{0.2, 0.4, 0.8}
\definecolor{bgreen}{rgb}{0.2, 0.6, 0.4}
\definecolor{bred}{rgb}{0.8, 0.4, 0.2}
\definecolor{bviolet}{rgb}{0.7, 0.2, 0.7}
\definecolor{blackred}{rgb}{0.6, 0.3, 0.3}
\definecolor{blackblue}{rgb}{0.3, 0.3, 0.6}
\definecolor{byellow}{rgb}{0.8, 0.6, 0.0}

\usepackage{tikz}
\usetikzlibrary{fit,arrows,trees,shapes,snakes,shapes.geometric,calc,decorations.markings}
\tikzset{
  treenode/.style = {align=center, inner sep=0pt, text centered,
    font=\sffamily},
  arnBleuPetit/.style = {treenode, circle, black, draw=bblue,
    fill=bblue!08,
    minimum width=0.8em, minimum height=0.5em
  },
  arnRougePetit/.style = {treenode, circle, black, draw=bred,
    fill=bred!20,
    minimum width=0.8em, minimum height=0.5em
  },
  arnVertPetit/.style = {treenode, circle, black, draw=bgreen,
    fill=bgreen!20,
    minimum width=0.8em, minimum height=0.5em
  },
  arn_n/.style = {treenode, circle, black, draw=bblue,
    text width=1.25em, very thick,
    fill=bblue!08},
  arn_r/.style = {treenode, circle, black, draw=bred,
    text width=1.15em, very thick,
    fill=red!08},
  arn_v/.style = {treenode, circle, bviolet, draw=bviolet,
    text width=1.5em, very thick,
    fill=bviolet!10},
  arn_g/.style = {treenode, circle, bgreen, draw=bgreen,
    text width=1.5em, very thick,
    fill=bblue!10},
}

\newcommandx*{\fromto}[7][3=black,4=black,5=0,6=0,7=.995,usedefault]{
    \path (#1) edge [thick,#4,bend right=#5,
    decoration={markings,mark=at position #7 with
    {\arrow[thick,#3, rotate=-#6]{>}}}, postaction={decorate}
    ] node {} (#2);
}
\newcommandx*{\sedge}[4][3=blackblue,4=0,usedefault]{
    \path (#1) edge [thick,#3,bend right=#4] node {} (#2);
}

\theoremstyle{definition}
\newtheorem{theorem}{Theorem}
\newtheorem{remark}{Remark}
\newtheorem{lemma}{Lemma}
\newtheorem{definition}{Definition}
\newtheorem{corollary}{Corollary}
\newtheorem{problem}{Problem}
\newtheorem{conjecture}{Conjecture}
\newtheorem{example}{Example}

\title{The birth of the contradictory component in random 2-SAT}

\author{Sergey Dovgal\thanks{%
The author is supported by the French project ANR project MetACOnc,
ANR-15-CE40-0014.
}}
\affil{
    \normalsize
  Institut Galilée,
  Université Paris 13,\\
  99 Avenue Jean Baptiste Clément 93430,\\
  Villetaneuse, France.\\
email: \texttt{dovgal-at-lipn.fr}
}

\date{\today}

\makeatletter
\newcommand{\oset}[3][0ex]{%
  \mathrel{\mathop{#3}\limits^{
    \vbox to#1{\kern-2\ex@
    \hbox{$\scriptstyle#2$}\vss}}}}
\makeatother

\def \n {\overline}

\newcommand{\fn}[1]{%
    \oset[-.3ex]{%
        \rule{5pt}{1pt}%
    }{%
        #1%
    }%
}

\def \leq {\leqslant}
\def \geq {\geqslant}
\DeclareMathOperator{\Link}{\boldsymbol \Lambda}
\DeclareMathOperator{\fA}{\mathcal A}
\DeclareMathOperator{\fB}{\mathcal B}
\DeclareMathOperator{\fC}{\mathcal C}

\DeclareMathOperator{\eLink}{\Lambda}
\def \eA {A}
\def \eB {B}

\usepackage{xspace} 
\newcommand*{\eg}{\textrm{e.g.}\@\xspace}
\newcommand*{\ie}{\textrm{i.e.}\@\xspace}

\begin{document}

\maketitle

\begin{abstract}
    We prove that, with high probability, the contradictory components of a
    random 2-SAT formula in the subcritical phase of the phase transition have
    only 3-regular kernels.
    This follows
    from the relation between these kernels and the complex component of a random
    graph in the subcritical phase. This partly settles the question about the
    structural
    similarity between the phase transitions in 2-SAT and random graphs. As a
    byproduct, we describe the technique that allows to obtain a full asymptotic
    expansion of the satisfiability in the subcritical phase.
    We also obtain the distribution of the number of contradictory variables and
    the structure of the spine in the subcritical phase.

\noindent \textbf{keywords.} 2sat, satisfiability, complex component, cores,
kernels, spine, analytic combinatorics, generating functions
\end{abstract}

\section{Introduction}

\subsection{Phase transitions}

Phase transitions in random graphs, directed graphs, hypergraphs, random
geometric complexes, percolations, real-world networks and constraint
satisfiability problems (CSP) have become rapidly developing areas with a huge
spread of interconnected publications.

The term ``phase transition'' was mostly used by physicists, but
can be also applied to many combinatorial situations, when a small change of a
certain parameter results in a huge asymptotic change of some other parameter.
The original studies of the physical phase transitions, including Ising and Potts
model, considered graphs which formed certain regular lattices: from rectangular
ones to more complicated including maps on surfaces. Of close
relation is the percolation theory which is sometimes called ``the simplest
model displaying a phase transition''.  Friendly introductions into percolation
theory and Potts models can be found in a PhD Thesis~\cite{dommers2013spin}, a
survey paper~\cite{beaudin2010little} and in a lecture
course~\cite{duminil2017lectures}.

Apart from the existing practical applications in hardware in software
engineering (\eg cuckoo
hashing~\cite{devroye2003cuckoo,dietzfelbinger2010tight}), phase transitions in
graphs and digraphs are studied in their own right.  Of the most
recent references on random graphs and networks can be mentioned the
books~\cite{janson2011random,van2016random}.
Concerning the phase transition in \emph{directed} graphs,
the width of the transitions window has been described
in~\cite{luczak2009critical} and a description of the giant
core is given in~\cite{pittel2017birth}.
Recent studies reveal the properties of phase transitions in random hypergraphs
\cite{de2015phase,cooley2018largest}
and simplicial complexes~\cite{cooley2018vanishing}.

One of the most recent
surveys on satisfiability is a part of the excellent ``Art of Computer
Programming, Volume 4'' by
Knuth~\cite{knuth2015art}. The k-SAT problem has also played its role on the
intersection of theoretical computer science and theoretical physics. One of the
surprising connections is that several NP-complete decision problems can be
reformulated in terms of positivity of the ground state energy of an Ising model
and that the techniques like belief propagation and cavity method used in statistical physics can be also used to predict
the behaviour of random
formulae, see~\cite{monasson1997statistical,biroli2002phase,mezard2002analytic}.
At the same time, phase transitions in random CSP seem to be closely related to
their algorithmic complexity, see~\cite{achlioptas2008algorithmic}.
While the existence of satisfiability threshold and its value for k-SAT with \(
k \geq 3 \) remains an unsolved open problem, several advances have been made,
including a precise asymptotic location of the transition point for \( k \to \infty
\)~\cite{coja2016asymptotic}, phase transitions of \( k
\)-XORSAT~\cite{herve2011random,pittel2016satisfiability},
and more
generally, systems of linear equations over \( F_q
\)~\cite{ayre2017satisfiability},
improvements in \( k \)-colorability threshold~\cite{coja2013chasing}.

The 2-SAT problem (see~\cref{definition:cnf}) is the easiest case of the more general \( k \)-SAT problem
which is NP-complete for \( k \geq 3 \), and admits a linear-time algorithm for
\( k = 2 \)~\cite{aspvall1982linear}. However, already the problem of maximising
the number of satisfying assignments of the 2-SAT, known as MAX-2-SAT is known
to be NP-complete as well. The study of the phase transition in 2-SAT and MAX-2-SAT
culminated in the papers~\cite{bollobas2001scaling}
and~\cite{coppersmith2004random}, where the critical width has been determined.
Kim~\cite{kim2008finding} improved
the bounds and provided an exact constant for the subcritical case.
Since then, it has been questioned whether the structural similarities between
the transitions in graphs, digraphs and 2-SAT exist.
As an example, it is known that for random graphs \( G_{n,m} \) with \( n \)
vertices and \( m = \frac{n}{2}(1 + \mu n^{-1/3}) \) edges,
\[
    \mathbb P(\text{%
        \( G_{n,m} \) contains only trees and unicycles%
    })
    = 1 - \dfrac{5 + o(1)}{24 |\mu|^3}
    \quad
    \text{ as } \mu \to -\infty
    \text{ with } n
    ,
\]
for a random 2-SAT formula \( F_{n,m} \) with \( n \) variables and
\( m = n (1 + \mu n^{-1/3}) \) clauses,
\[
    \mathbb P(\text{%
        \( F_{n,m} \) is satisfiable%
    })
    = 1 - \dfrac{1 + o(1)}{16 |\mu|^3}
    \quad
    \text{ as } \mu \to -\infty
    \text{ with } n
    ,
\]
and for random directed graphs \( D_{n,m} \) with \( n \) vertices and \( m =
n(1 + \mu n^{-1/3}) \) oriented edges,
\[
    \mathbb P\left(\substack{
        \text{%
        every strong component in \( D_{n,m} \) is either}
        \\
        \text{a vertex or a cycle of
        length \( O(n^{1/3} \mu^{-1}) \)%
    }%
    }\right)
    \to 1
    \quad
    \text{ as } \mu \to -\infty
    \text{ with } n
    ,
\]
where \( o(1) \) goes to zero as \( |\mu| \ll n^{1/3} \) (the formula for random
graphs requires \( |\mu| \ll n^{1/12} \) and can be further improved to
\( |\mu|\ll n^{1/3} \), see~\cite{herve2011random}). The factor \( 5/24 \) is obtained by
adding the inverse numbers of automorphisms \( 1/12 \) and \( 1/8 \) of
the two possible cubic (\ie 3-regular) multigraphs with \( 3 \) edges and \( 2
\) vertices,
which appear as the only two possible cores at the point when a graph doesn't
anymore consist only of trees and unicycles,
see~\cref{subsection:sp:lemma:random:graphs}
and~\cref{remark:bicyclic:components} for an
explanation of this phenomenon.
A different viewpoint on the giant component and satifsfiability has been
proposed in~\cite{molloy2008does}.
Clearly, the structure of a critical implication digraph might not be similar to
that of a critical simple graph or a critical digraph, due to the presence of
certain symmetries, which have been emphasised in~\cite{kravitz2007random}.
Nevertheless, it still seems
to be possible to draw certain structural similarities between the models.

We follow the approach of \emph{analytic combinatorics} which was used to obtain
the structure of random graph at the point of the phase
transition~\cite{flajolet1989first,janson1993birth,flajolet2004airy}.
The approach makes use of the generating function technique which gives a very
clear structural vision of the corresponding combinatorial objects. In addition
to the aestethic benefits (see \eg a unifying approach for the upper bounds on
satisfiability threshold~\cite{puyhaubert2004generating}), it allows to give
very precise asymptotic descriptions, often accompanied by complete asymptotic
expansions. In particular, inside the transition window, the probability that a
random graph doesn't contain a complex component (\ie consists only of trees and
unicycles) has been expressed in terms of Airy function which has appeared in
many other contexts in analytic
combinatorics~\cite{banderier2001random,flajolet2004airy}.
Another interpretation, using the fact that Airy function is linked to the area
under a Brownian motion, was discovered by Aldous~\cite{aldous1997brownian}, see
also~\cite{addario2012continuum}.
The enumeration of unsatisfiable 2-SAT formulae is equivalent to
forbidding a certain contradictory pattern
(namely, the contradictory circuit, see~\cref{subsection:definitions}).
Of the closest to our approach here is~\cite{collet2018threshold} where the
authors use the analytic approach to study the containment problem for small
subgraphs.
Applying the analytic methods to the phase transitions of SAT-formulae and
directed graphs remains one of the important challenging problems.

The \emph{spine} of a 2-CNF (see~\cref{definition:contradictions})
has been introduced in~\cite{bollobas2001scaling}
as a useful tool for combinatorial analysis of the probability of
satisfiability inside the critical window. Originally, the spine is defined as
the set of variables that are forced to take the FALSE value in any satisfying
assignment.
Later, the spine has been shown to impact the complexity of the underlying decision
problems in a more general setting for various constraint satisfaction problems
including \( k \)-XORSAT, graph bipartition problem, random
3-coloring~\cite{solymosispine,istrate2005spines}.  In our analysis of
satisfiability we do not use the properties of the spine, but we find this
parameter interesting by itself. We obtain a new structural result about the
spine of a random formula.

\subsection{Definitions}
\label{subsection:definitions}

A \emph{simple graph} \( G = (V, E) \), \( E \subset
\{ \{x, y\} \mid x, y \in V,\ x \neq y \} \) is a graph
without loops and multiple edges, so that every
edge connects a set of two distinct vertices and for every pair of vertices
there is at most one edge connecting them. Contrary to a simple graph, a
\emph{multigraph} may contain an arbitrary number of loops and multiple edges.
By \( \mathcal G(n,m) \) we denote the set of all simple graphs with \( n \)
vertices and \( m \) edges. We say that a graph has \emph{size} \( n \) if it
contains \( n \) vertices.

A \emph{simple digraph} is a synonym for a \emph{simple directed graph} which we
define as a pair \( D = (V, E) \) of the set of vertices \( V \) and the set of
edges \( E \subset \{ (x, y) \mid x, y \in V,\ x \neq y \} \) such that every
directed edge \( x \to y \) is represented by
a pair of vertices \( (x, y) \);
a simple digraph contains no loops \( x \to x \), and no multiple edges;
at the same time, any two distinct vertices \( x \) and \( y \) can have
simultaneously both directed edges \( x \to y \) and \( y \to x \) between them.
An \emph{unoriented projection} of a simple digraph is a multigraph (possibly
containing double edges) which is obtained by dropping the orientations of the
edges.  We denote by \( \mathcal D(2n, m) \) the set of simple digraphs with \(
2n \) vertices and \( m \) oriented edges.

Below we give the definitions related to the 2-SAT model which contains the
formulae in \emph{2-conjunctive normal form} (2-CNF).

\begin{definition}[Variables, literals and clauses]
    \label{definition:cnf}
Consider  \( n \) \emph{Boolean variables}
\( \{x_1, x_2, \ldots x_n\} \), so that \( x_i \in \{0, 1\} \) for all \( i \in
\{1, \ldots, n\} \). The logical negation of \( x_i \) is denoted by \( \n x_i \).
Each of \( \{x_i, \n x_i\} \) is called a \emph{literal}, while the two literals
\( \{x_i, \n x_i\} \) refer to the same \emph{variable} \( x_i \).
We say that two literals \( \xi \) and \( \eta \) are
\emph{complementary} if \( \xi = \n \eta \).
Two literals \( \xi \) and \( \eta \) are said to be \emph{strictly distinct} if
their underlying variables are distinct.
The 2-\emph{clauses} are disjunctions of two
literals, corresponding to distinct Boolean variables, in other words, each
2-clause is of the form \( (\xi_j \lor \eta_j) \) where \( \xi_j \) and \( \eta_j
\) belong to the set of \( 2n \) possible literals
\( \{ x_1, \ldots, x_n, \n x_1, \ldots, \n x_n\} \).
We do not distinguish clauses obtained by a change of the order of variables
inside a disjunction, so \( (\xi_j \lor \eta_j) \equiv (\eta_j \lor \xi_j) \).
A \emph{2-CNF} is a conjunction of clauses,
\ie a formula of the form \( \land_{j = 1}^m (\xi_j \lor \eta_j) \) where
each of the clauses \( (\xi_j \lor \eta_j) \) is distinct for \( j \in \{ 1,
\ldots, m \} \). A formula is \emph{satisfiable} (SAT) if there exists variable
assignment yielding the truth value of the formula.
By \( \mathcal F(n, m) \) we denote the set of 2-CNF formulae with \( n \)
Boolean variables and \( m \) clauses.
\end{definition}

In different models, the
distinctness condition for clauses is not required, but we show that it is not essential for
the phase transition and for our techniques. More explicitly, we require for
technical reasons that in
a 2-CNF formula in our model the conditions
\ref{condition:no:double:edges}--\ref{condition:no:multigraphs} hold:
\begin{enumerate}[label=(C\arabic*)]
    \item \label{condition:no:double:edges}
        clauses \( (x_i \lor x_i) \) are not allowed;
    \item \label{condition:no:loops}
        clauses \( (x_i \lor \n x_i) \) are not allowed;
    \item \label{condition:no:multigraphs}
        all the clauses are distinct.
\end{enumerate}
Each of the conditions
\ref{condition:no:double:edges}--\ref{condition:no:multigraphs} may be
violated without changing the main results, see~\cref{remark:breaking:conditions}.

\begin{definition}[Implication digraphs]
Any 2-CNF with \( n \) Boolean variables and \( m \) clauses can be represented
in the form of an \emph{implication digraph} where every clause \( (x \lor y) \)
is replaced with two directed edges \( (\n x \to y) \) and \( (\n y \to x) \).

If there is a directed path from a vertex \( x \) to a vertex \( y \) in a
directed graph \( D \), we write \( x \rightsquigarrow y \).
In the case when \( D \) is an implication digraph,
we also say that a literal \( x \) \emph{implies} literal \( y \).
\end{definition}

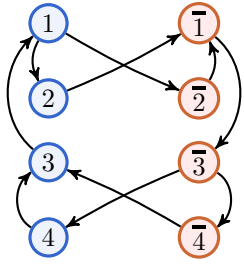
\begin{figure}[hbt]
\centering
\begin{tikzpicture}[>=stealth',thick, node distance=1.0cm] 
\draw
node[arn_n](1)  {\( 1\)}
node[arn_n](2)  [below of=1] {\( 2 \)}
node[arn_n](3)  [below of=2, node distance=.85cm] {\( 3 \)}
node[arn_n](4)  [below of=3] {\( 4 \)}
node[arn_r](1') [right of=1, node distance=2cm] {\( \fn 1 \)}
node[arn_r](2') [right of=2, node distance=2cm] {\( \fn 2 \)}
node[arn_r](3') [right of=3, node distance=2cm] {\( \fn 3 \)}
node[arn_r](4') [right of=4, node distance=2cm] {\( \fn 4 \)}
;
\fromto{1}{2'}[][][4]
\fromto{2}{1'}[][][4]
\fromto{3'}{4}[][][4]
\fromto{4'}{3}[][][4]
\fromto{1}{2}[][][ 30][10]
\fromto{2'}{1'}[][][ 30][10]
\fromto{1'}{3'}[][][-50][-5]
\fromto{3}{1}[][][-50][-5]
\fromto{4}{3}[][][-60][-15]
\fromto{3'}{4'}[][][-60][-15]
\end{tikzpicture}
\caption{\label{fig:implication:digraph}Implication digraph.}
\end{figure}

\begin{definition}[Spine, contradictory variables and circuits]
    \label{definition:contradictions}
    The \emph{spine} of a formula \( F \) (denoted \( \mathcal S(F) \)) is the
    set of literals that imply their complementary literals, \ie
    \[
        \mathcal S(F) := \{
            x \mid
            x \rightsquigarrow \n x
            \text{ in $F$}
        \}.
    \]
    A variable \( x \) is \emph{contradictory} if \( x \in \mathcal S(F) \) and
    \( \n x \in \mathcal S(F) \). The \emph{contradictory component}
    \( \mathcal C(F) \)
    of a
    formula \( F \) is then formed of all the contradictory variables, \ie
    \[
        \mathcal C(F) := \{
            x \mid
            x \rightsquigarrow \n x \rightsquigarrow x
            \text{ in $F$}
        \}
        .
    \]
    A \emph{contradictory circuit} is a distinguished directed path passing
    through \( x \), \( \n x \) and \( x \). It can be easily shown that any
    variable belonging to a contradictory circuit is also contradictory
    (possibly via a different contradictory circuit).
    It is well known and has been proven in~\cite{aspvall1982linear}, that a
    formula is satisfiable if and only if it contains a so-called
    \emph{contradictory circuit}, that is, there exists a literal \( x \) such that
    \( x \rightsquigarrow \n x \) and \( \n x \rightsquigarrow x \).
    Consequently, a formula is satisfiable if and only if its contradictory
    component is empty.
\end{definition}

\begin{example}
\cref{fig:implication:digraph} contains an implication digraph of a
formula \( (\n x_1 \lor \n x_2)(\n x_1 \lor x_2)(x_1 \lor \n x_3)(x_3 \lor
x_4)(x_3 \lor \n x_4) \). This formula is unsatisfiable because its implication
digraph containts a contradictory circuit \( 1 \to \n 2 \to \n 1 \to \n 3 \to 4
\to 3 \to 1\) passing through \( 1 \) and \( \n 1 \).
\end{example}

\begin{remark}
\label{remark:set:cscc}
    A contradictory component forms a set of strongly connected components such
    that there is no directed path between any two strongly connected
    components. Indeed, if \( \mathcal C_1 \) and \( \mathcal C_2 \) are such
    strongly connected components (where each variable is contradictory), and
    there is a path from \( x \in \mathcal C_1 \) to \( y \in \mathcal C_2 \),
    then, the complementary path \( \n y \rightsquigarrow \n x \) is also a part
    of the implication digraph, while \( \n y \in \mathcal C_2 \), \( \n x
    \in \mathcal C_1 \), since each of the strongly connected components is
    contradictory. Therefore, the presence of a directed path from \( \mathcal
    C_1 \) to \( \mathcal C_2 \) implies a directed path from \( \mathcal C_2 \)
    to \( \mathcal C_1 \) and they form a single strongly connected component.
\end{remark}

\subsection{Structure of the paper}
The paper is structures as follows. In~\cref{section:sum:representation} we
introduce the concept of \emph{sum-representation} which is used throughout the
paper. In this section, the main comparisons between simple graphs and
implication digraphs of \( 2 \)-SAT are drawn, and the \emph{excess} of a
contradictory component is introduced.
In~\cref{section:generating:functions} the classical symbolic method for
generating functions of simple graphs is explained, along with its variations for
weakly connected directed graphs. In this section, the concept of compensation
factor of a contradictory component is introduced.

In~\cref{section:asymptotic:analysis}, the actual asymptotic analysis is done.
The main results of this paper are formed
into~\cref{theorem:satisfiability,theorem:contradictory:variables,theorem:spine}.
In~\cref{theorem:satisfiability} we prove that in the subcritical phase of the
2-SAT, the kernels of the
contradictory components are typically cubic and that the asymptotic expansion
of the probability of satisfiability is linked to the compensation factors of
such cubic components.
In~\cref{theorem:contradictory:variables} we prove that the number
of contradictory variables, scaled by \( n^{1/3} |\mu|^{-1} \)
follows a mixture of Gamma laws, with the first nontrivial case being
Gamma(3) with a scale parameter \( \mu^{-1} n^{1/3} \) corresponding to the
contradictory component of minimal possible excess \( 1 \).
Finally, in~\cref{theorem:spine} we classify the spine literals according to
the multiplicities of their paths to the complementary variables, and prove that
a negligible part of the spine can be removed in such a way that the remaining
literals form a set of tree-like structures.

Several naturally arising open problems and
conjectures are described in~\cref{section:open:problems}, along with some
remarks on how the results presented in this paper could be potentially extended.
In the appendix, in~\cref{section:full:asymptotic:expansion} we resent a ``proof
of concept'' full asymptotic expansion of the saddle point lemma which allows in
principle to construct the full asymptotic expansion of the satisfiability.

\section{Sum-representation of implication digraphs}
\label{section:sum:representation}

For the study of the 2-SAT phase transition, it is convenient to consider
digraphs with an even number of vertices \( 2n \) with a special vertex
labelling convention. Instead of the labels
\( \{n+1, n+2, \ldots, 2n\} \) we conventionally use the synonyms
\( \{ \n 1, \n 2, \ldots, \n n \} \). Under this re-assignment, the nodes with
labels \( \{1, 2, \ldots, n\} \) correspond to Boolean literals \( \{x_1, x_2,
\ldots, x_n \} \), and the nodes with labels \( \{\n 1, \ldots, \n n\} \)
correspond to the negations of these literals \( \{\n x_1, \ldots, \n x_n \} \).
Since the complementary of \( \n x_i \) is \( x_i \), the same applies to the
labels, \( \n{\n i} = i \).

\begin{definition}[Conflicts in digraphs]
We define a \emph{complementary} of the edge \(x \to y\) as \(\n y \to \n x\).
We shall say that a pair of complementary edges in a directed graph form a
\emph{conflict}.
We say that a digraph is \emph{conflict-free} if there are no conflicting pairs
of edges inside it, \ie if no two edges are complementary.
\end{definition}

Note that in the model that we consider, a random 2-CNF formula contains
neither clauses of type \( (x \lor x) \) nor of \( (x \lor \n x) \). This
implies that the underlying implication digraph contains neither loops nor
multiple edges.

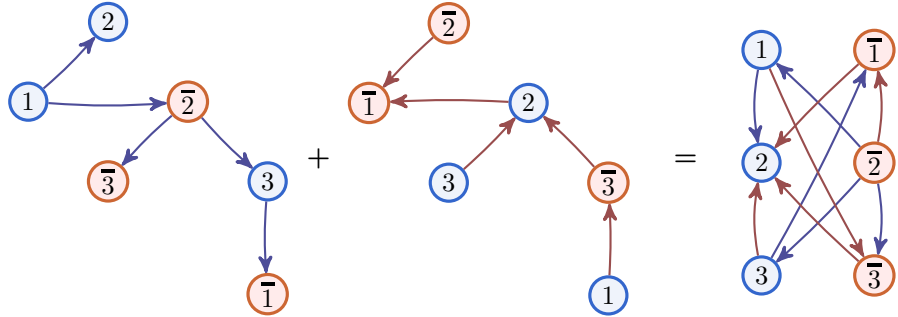
\begin{figure}[hbt]
\centering
\raisebox{-2cm}{
\begin{tikzpicture}[>=stealth',thick, node distance=1.5cm]
\draw
node[arn_n](2)                       {\(    2 \)}
node[arn_n](1)  [below left of  =2 ] {\(    1 \)}
node[arn_r](2') [below right of =2 ] {\( \fn 2 \)}
node[arn_r](3') [below left of  =2'] {\( \fn 3 \)}
node[arn_n](3)  [below right of =2'] {\(    3 \)}
node[arn_r](1') [below of       =3 ] {\( \fn 1 \)}
;
\fromto{1}{2}  [blackblue, very thick][blackblue, thick][4][1]
\fromto{1}{2'} [blackblue, very thick][blackblue, thick][4][1]
\fromto{2'}{3'}[blackblue, very thick][blackblue, thick][4][1]
\fromto{2'}{3} [blackblue, very thick][blackblue, thick][4][1]
\fromto{3}{1'} [blackblue, very thick][blackblue, thick][4][1]
\end{tikzpicture}
}
\( \boldsymbol + \)
\raisebox{-2cm}{
\begin{tikzpicture}[>=stealth',thick, node distance=1.5cm]
\draw
node[arn_r](2')                       {\( \fn 2 \)}
node[arn_r](1')  [below left of  =2' ] {\( \fn 1 \)}
node[arn_n](2) [below right of =2' ] {\(    2 \)}
node[arn_n](3) [below left of  =2] {\(    3 \)}
node[arn_r](3')  [below right of =2] {\( \fn 3 \)}
node[arn_n](1) [below of       =3' ] {\(    1 \)}
;
\fromto{2'}{1'}[blackred, very thick][blackred, thick][3][1]
\fromto{2}{1'} [blackred, very thick][blackred, thick][3][1]
\fromto{3}{2}  [blackred, very thick][blackred, thick][3][1]
\fromto{3'}{2} [blackred, very thick][blackred, thick][3][1]
\fromto{1}{3'} [blackred, very thick][blackred, thick][3][1]
\end{tikzpicture}
}
\(\quad \boldsymbol = \quad \)
\raisebox{-1.75cm}{
\begin{tikzpicture}[>=stealth',thick, node distance=1.5cm]
\draw
node[arn_n](1)               {\(    1 \)}
node[arn_n](2)  [below of=1] {\(    2 \)}
node[arn_n](3)  [below of=2] {\(    3 \)}
node[arn_r](1') [right of=1] {\( \fn 1 \)}
node[arn_r](2') [right of=2] {\( \fn 2 \)}
node[arn_r](3') [right of=3] {\( \fn 3 \)}
;
\fromto{1}{2}  [blackblue, very thick][blackblue, thick][10]
\fromto{2'}{1} [blackblue, very thick][blackblue, thick][3]
\fromto{2'}{3'}[blackblue, very thick][blackblue, thick][-10]
\fromto{2'}{3} [blackblue, very thick][blackblue, thick][-3]
\fromto{3}{1'} [blackblue, very thick][blackblue, thick][3]

\fromto{2'}{1'}[blackred, very thick][blackred, thick][10]
\fromto{1'}{2} [blackred, very thick][blackred, thick][3]
\fromto{3}{2}  [blackred, very thick][blackred, thick][-10]
\fromto{3'}{2} [blackred, very thick][blackred, thick][-3]
\fromto{1}{3'} [blackred, very thick][blackred, thick][3]
\end{tikzpicture}
}
\caption{\label{fig:sum:representation}Example of a sum representation digraph
\( G \) and its complementary \( \n G \) whose edge-union \( G + \n G \) gives
an implication digraph.}
\end{figure}
\begin{definition}[Sum-representation]
    A digraph \( G = (V, E) \) with \( 2n \) conventionally labelled vertices
    \( V = \{1, \cdots, n, \n 1, \cdots, \n n\} \) and \( |E| = m \) edges is
    called a \emph{sum-representation digraph} if it does not contain loops,
    multiple edges, edges of type \( x \to \n x \) and is conflict-free.
    The \emph{complementary} digraph \( \n
    G \) of a sum-representation digraph is obtained by replacing all the edges
    by their respective complementaries. We say that \( G \) is a
    \emph{sum-representation of} an implication graph \( G + \n G \) which is a
    digraph obtained by joining the
    sets of edges of \( G \) and \( \n G \) (see~\cref{fig:sum:representation}).
    Every implication digraph with \( 2m \) edges has \( 2^m \)
    sum-representations, since for each of the \( m \) clauses of the corresponding
    formula, \( 2 \) choices of complementary edges are available.
    We denote the set of all sum-representation digraphs with \( 2n \) vertices
    and \( m \) oriented edges as \( \mathcal D^\circ(2n,m) \).
\end{definition}

Instead of implication digraphs with \( 2n \) vertices and \( 2m \) oriented
edges we enumerate conflict-free digraphs with \( 2n \) vertices and \( m \)
edges. It is worth noticing that \( |\mathcal F(n,m)| = 2^{-m} |\mathcal
D^\circ(2n,m)| \). Note also that \( \mathcal D^\circ(2n,m) \subset \mathcal
D(2n,m) \).

An \emph{edge rotation} is a process of transformation of an edge \( x \to y \)
into its complementary edge \( \n y \to \n x \) (see~\cref{fig:edge:rotation}).
We say that \( \pi_1 \) is \emph{equivalent} to \( \pi_2 \) if it can be
obtained by a sequence of edge rotations. This means that \( \pi_1 \)
and \( \pi_2 \) are both sum-representation of the same implication digraph.

\begin{figure}[hbt]
\centering
\raisebox{-2cm}{
\begin{tikzpicture}[>=stealth',thick, node distance=1.5cm]
\draw
node[arn_n](2)                       {\(    2 \)}
node[arn_n](1)  [below left of  =2 ] {\(    1 \)}
node[arn_r](2') [below right of =2 ] {\( \fn 2 \)}
node[arn_r](3') [below left of  =2'] {\( \fn 3 \)}
node[arn_n](3)  [below right of =2'] {\(    3 \)}
node[arn_r](1') [below of       =3 ] {\( \fn 1 \)}
;
\fromto{1}{2}  [very thick][thick][4][1]
\fromto{1}{2'} [very thick][thick][4][1]
\fromto{2'}{3'}[very thick][thick][4][1]
\fromto{2'}{3} [very thick][thick][4][1]
\fromto{1'}{3} [bblue, very thick][bblue, ultra thick][4][1]
\end{tikzpicture}
}
\( \boldsymbol \Rightarrow \)
\raisebox{-2cm}{
\begin{tikzpicture}[>=stealth',thick, node distance=1.5cm]
\draw
node[arn_n](2)                       {\(    2 \)}
node[arn_n](1)  [below left of  =2 ] {\(    1 \)}
node[arn_r](2') [below right of =2 ] {\( \fn 2 \)}
node[arn_r](3') [below left of  =2'] {\( \fn 3 \)}
node[arn_n](3)  [below right of =2'] {\(    3 \)}
node[arn_r](1') [below of       =3 ] {\( \fn 1 \)}
;
\fromto{1}{2}  [very thick][thick][4][1]
\fromto{1}{2'} [very thick][thick][4][1]
\fromto{2'}{3'}[very thick][thick][4][1]
\fromto{2'}{3} [very thick][thick][4][1]
\fromto{3'}{1} [bred, very thick][bred, ultra thick][-4][-1]
\end{tikzpicture}
}
\caption{\label{fig:edge:rotation}Example of two equivalent sum-representations
    obtaine by one edge rotation.
An edge \( \n 1 \to 3 \) is replaced by its complementary \( \n 3 \to 1 \).
}
\end{figure}
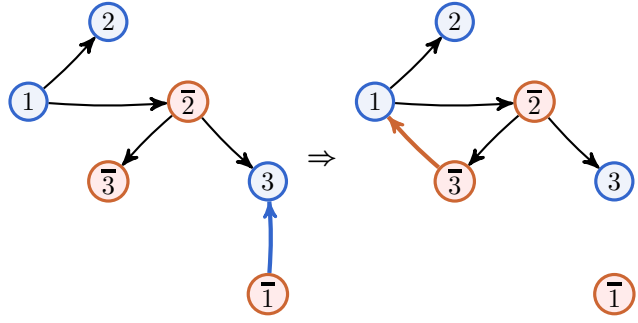

\subsection{The spine}
\label{subsection:sum:representation:spine}

Let \( y \) be a literal belonging to the spine of some formula \( F \in
\mathcal F(n,m) \), which means that \( y
\rightsquigarrow \n y \) in the implication digraph corresponding to \( F \).
Counting such literals \( y \) with the
multiplicities of the corresponding paths \( y \rightsquigarrow \n y \) gives a larger number
than just the cardinality of the spine, however, we are going to show that in
the subcritical phase, \ie when \( m = n(1 + \mu n^{-1/3}) \), \( \mu \to -\infty
\), it gives asymptotically the same result,
see~\cref{theorem:spine,corollary:spine}.
The inherent reason behind this is
that the majority of the spine components form certain tree-like structures
(see~\cref{fig:spine:tree}).

Let
us say that a path \( y \rightsquigarrow \n y \) is \emph{strictly distinct} if
all the vertices of the path, except \( y \) and \( \n y \), are pairwise
strictly distinct.
If \( y \rightsquigarrow \n y \), it is not always possible to find a strictly
distinct directed path from \( y \) to \( \n y \), but it is possible to split
it into a few sections, each strictly distinct.

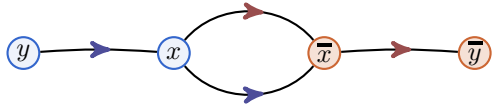
\begin{figure}[hbt]
    \centering
\begin{tikzpicture}[>=stealth',thick, node distance=1.0cm]
\draw
node[arnBleuPetit, minimum height = 1.2em,  thick](x)
                    at (0,0) {\( x \)}
node[arnRougePetit, minimum height = 1.2em, thick](x')
                    at (2,0) {\( \fn x \)}
node[arnBleuPetit, minimum height = 1.2em,  thick](y)
                    at (-2,0){\( y \)}
node[arnRougePetit, minimum height = 1.2em, thick](y')
                    at (4,0) {\( \fn y \)}
;
\fromto{x}{x'}     [ultra thick,blackred] [thick]
                    [-50][-5][0.6];
\fromto{x}{x'}     [ultra thick,blackblue][thick]
                    [50][5][0.6];
\fromto{x'}{y'}     [ultra thick,blackred][thick]
                    [-4][-1][0.6];
\fromto{y}{x}     [ultra thick,blackblue] [thick]
                    [-4][-1][0.6];
\end{tikzpicture}
\caption{\label{fig:spine:sd}A directed path from \( y \) to \( \n y \) is split
    into three strictly distinct sections.}
\end{figure}

\begin{lemma}[Minimal spinal paths]
    \label{lemma:minimal:spinal:path}
Let \( y \rightsquigarrow \n y \) inside an implication digraph. Then there
exists a literal \( x \) (not necessarily disctinct from \( y \)) such that \( y
\rightsquigarrow x \rightsquigarrow \n x \rightsquigarrow \n y \),
the paths \( y \rightsquigarrow x \) and \( x \rightsquigarrow \n x \) are
strictly distinct and do not intersect with each other.

As depicted in~\cref{fig:spine:sd},
each arrow represents a directed path
with strictly distinct literals. Since \( x \rightsquigarrow \n x \) is strictly
distinct, it has no intersection with its complementary path.
It is convenient to denote by \( x \rightsquigarrow_1 \n x \) and \( x
\rightsquigarrow_2 \n x \) the two complementary versions of the path \( x
\rightsquigarrow \n x \).
We shall call every such quadruple
\( (y \rightsquigarrow x, x \rightsquigarrow_1 \n x, x \rightsquigarrow_2 \n x, \n x
\rightsquigarrow \n y)\) a \emph{minimal spinal path}.
\end{lemma}

\begin{proof}
    Without loss of generality assume that the path \( y \rightsquigarrow \n y
    \) does not pass through the same vertex twice.
    Assume that the path \( y \rightsquigarrow \n y \) is formed as a sequence
    of edges \( y = a_0 \to a_1 \), \( a_1 \to a_2 \), \ldots, \( a_{\ell - 1}
    \to a_\ell = \n y \). If the path is strictly distinct, we can set \( y = x
    \) and the lemma is proven. Otherwise we choose \( x = a_k
    \) where \( k \) is the minimal index such that both \( a_k \) and
    \( \n a_k \) belong to the path \( y \rightsquigarrow \n y \), and the path
    \( x \rightsquigarrow \n x \) is strictly distinct. It is always possible to
    choose such an index: if the original path is not strictly distinct, we keep
    choosing a smaller directed path connecting some literal \( x \) and its
    complementary, until we obtain a path of length at most \( 2 \) for which
    the statement is obviously true since the implication digraph does not
    contain loops and multiple edges.
    Note that under such a choice, the path \( y \rightsquigarrow \n x \) is
    also strictly distinct (except for the vertices \( x \) and \( \n x \)),
    because otherwise it is possible to choose a smaller \( k \). The
    corresponding part \( \n x \rightsquigarrow \n y \) can be constructed as a
    complementary of the path \( y \rightsquigarrow x \).
\end{proof}

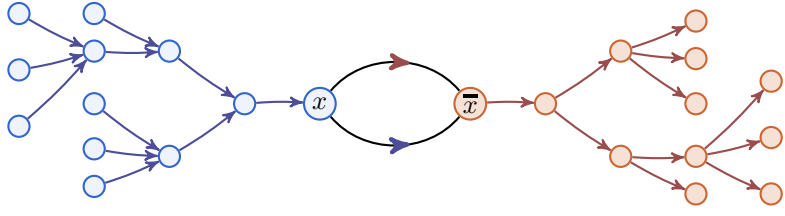
\begin{figure}[hbt]
\centering
\begin{tikzpicture}[>=stealth',thick, node distance=1.0cm]
\draw
node[arnBleuPetit, minimum height = 1.2em,  thick](x)
                    at (0,0) {\( x \)}
node[arnRougePetit, minimum height = 1.2em, thick](x')
                    at (2,0) {\( \fn x \)}
node[arnBleuPetit](q1)  at (-1,0) {}
node[arnBleuPetit](q11)  at ($(q1) + (-1,-0.7)$) {}
node[arnBleuPetit](q111)  at ($(q11) + (-1,-0.4)$) {}
node[arnBleuPetit](q112)  at ($(q11) + (-1,0.1)$) {}
node[arnBleuPetit](q113)  at ($(q11) + (-1,+0.7)$) {}
node[arnBleuPetit](q12)  at ($(q1) + (-1,0.7)$) {}
node[arnBleuPetit](q121)  at ($(q12) + (-1,+0.5)$) {}
node[arnBleuPetit](q122)  at ($(q12) + (-1,0)$) {}
node[arnBleuPetit](q1221)  at ($(q122) + (-1,-1)$) {}
node[arnBleuPetit](q1222)  at ($(q122) + (-1,-0.25)$) {}
node[arnBleuPetit](q1223)  at ($(q122) + (-1,0.5)$) {}
node[arnRougePetit](w1) at (3,0)  {}
node[arnRougePetit](w11)   at ($(w1)   - (-1,-0.7)$) {}
node[arnRougePetit](w111)  at ($(w11)  - (-1,-0.4)$) {}
node[arnRougePetit](w112)  at ($(w11)  - (-1,0.1)$) {}
node[arnRougePetit](w113)  at ($(w11)  - (-1,+0.7)$) {}
node[arnRougePetit](w12)   at ($(w1)   - (-1,0.7)$) {}
node[arnRougePetit](w121)  at ($(w12)  - (-1,+0.5)$) {}
node[arnRougePetit](w122)  at ($(w12)  - (-1,0)$) {}
node[arnRougePetit](w1221) at ($(w122) - (-1,-1)$) {}
node[arnRougePetit](w1222) at ($(w122) - (-1,-0.25)$) {}
node[arnRougePetit](w1223) at ($(w122) - (-1,0.5)$) {}
;
\fromto{x}{x'}        [ultra thick,blackred] [thick][-50][-5][0.6];
\fromto{x}{x'}        [ultra thick,blackblue][thick][50][5][0.6];
\fromto{x'}{w1}       [thick][blackred][-4][-1];
\fromto{q1}   {x}     [thick][blackblue][-4][-1];
\fromto{q11}  {q1}    [thick][blackblue][4][1];
\fromto{q111} {q11}   [thick][blackblue][4][1];
\fromto{q112} {q11}   [thick][blackblue][4][1];
\fromto{q113} {q11}   [thick][blackblue][4][1];
\fromto{q12}  {q1}    [thick][blackblue][4][1];
\fromto{q121} {q12}   [thick][blackblue][4][1];
\fromto{q122} {q12}   [thick][blackblue][4][1];
\fromto{q1221}{q122}  [thick][blackblue][4][1];
\fromto{q1222}{q122}  [thick][blackblue][4][1];
\fromto{q1223}{q122}  [thick][blackblue][4][1];
\fromto{w1}   {w11}   [thick][blackred][4][1];
\fromto{w11}  {w111}  [thick][blackred][4][1];
\fromto{w11}  {w112}  [thick][blackred][4][1];
\fromto{w11}  {w113}  [thick][blackred][4][1];
\fromto{w1}   {w12}   [thick][blackred][4][1];
\fromto{w12}  {w121}  [thick][blackred][4][1];
\fromto{w12}  {w122}  [thick][blackred][4][1];
\fromto{w122} {w1221} [thick][blackred][4][1];
\fromto{w122} {w1222} [thick][blackred][4][1];
\fromto{w122} {w1223} [thick][blackred][4][1];
\end{tikzpicture}
\caption{\label{fig:spine:tree}For every literal \( y \) in a
tree-like spine structure, there is a unique path \( y \rightsquigarrow \n y \).}
\end{figure}

Let \( \sigma \) be the
random variable denoting the cardinality of the spine in a random formula
\( F \in \mathcal F(n,m) \). Then, since the number of sum-representations of a
single implication digraph with \( 2m \) edges is \( 2^m \),
the expected value of \( \sigma \) can be expressed as
\[
    \mathbb E \sigma = \dfrac{
        \left|\left\{
            (G, x) \mid
                G \in \mathcal D^\circ(2n, m)
                ,\
            \text{$x \rightsquigarrow \n x$ in $G+\n G$}
        \right\}\right|
    }{
        | \mathcal D^\circ(2n, m) |
    }
    .
\]
Since the spine literals can be characterised as the literals for which there
\emph{exists} a path connecting a literal and its complementary, the counting of
such literals can be handled though inclusion-exclusion.
Accordingly,
\begin{align*}
    \mathbb E \sigma &=
    \dfrac{
        \left|\left\{
            (G, y, p) \ | \
                G \in \mathcal D^\circ(2n,m)
                ;\
            p \text{ is a minimal spinal path \( y \rightsquigarrow \n y \)
                in \( G + \n G \)}
        \right\}\right|
    }{
        | \mathcal D^\circ(2n, m) |
    }
    \\
    & -
    \dfrac{
        \left|\left\{
            (G, y, p_1, p_2) \ | \
                G \in \mathcal D^\circ(2n,m)
                ;\
                p_1,\text{ } p_2 \text{ are distinct }
                \text{minimal spinal paths \( y \rightsquigarrow \n y \) in \( G
                + \n G \) }
        \right\}\right|
    }{
        | \mathcal D^\circ(2n, m) |
    }
    + \cdots
    .
\end{align*}

Consider a pattern \( p \in G \), \( G \in \mathcal D^\circ(2n, m) \) consisting
of two paths \( y \rightsquigarrow x \), \( x \rightsquigarrow \n x \), where
all the literals on the two paths are pairwise stictly distinct (except for
\( x \) and \( \n x \)). Suppose that the pattern \( p \) has total length \(
\ell \).
Then, among \( 2^m \) equivalent sum-representations of \( G \),
\( 2^{m-\ell} \) sum-representations contain the
pattern \( p \) unaltered.
Among \( 2^\ell \) possible combinations of edge
rotations of one of the \( \ell \) edges of \( p \) only, there are exactly \( 2
\) possibilities that result in the same type of pattern, the second one
obtained by converting the path \( x \rightsquigarrow \n x \) into its
complementary.

Therefore, the number of sum-representation digraphs with a distinguished
pattern \( p = y \rightsquigarrow x \rightsquigarrow \n x \) of length \( \ell \)
counted with multiplicity \( 2^\ell \) enumerates (with
multiplicity 2-to-1) implication digraphs with a distinguished
minimal spinal path. This allows us to rewrite the first summand of \( \mathbb E
\sigma \) as
\begin{align*}
    &
    \dfrac{
        \left|\left\{
            (G, y, p) \ | \
                G \in \mathcal D^\circ(2n,m)
                ;\
            p \text{ is a minimal spinal path \( y \rightsquigarrow \n y \)
                in \( G + \n G \)}
        \right\}\right|
    }{
        | \mathcal D^\circ(2n, m) |
    }
    \\ &
    =
    \dfrac12
    \cdot
    \dfrac{
        \sum_{\ell}
        2^\ell
        \left|\left\{
            (G, y, x, p) \ | \
                G \in \mathcal D^\circ(2n,m)
                ;\
                p \text{ is a distinguished pattern \( y \rightsquigarrow x
                    \rightsquigarrow \n x \)
                    of length \( \ell \)
                in \( G \)}
        \right\}\right|
    }{
        | \mathcal D^\circ(2n, m) |
    }
    .
\end{align*}
All the subsequent summands of \( \mathbb E \sigma \) can be rewritten in the same manner as well, though it
requires considering several cases for the mutual configuration of \( p_1 \) and
\( p_2 \) and using possibly different factors instead of \( \frac12 \) for
the multiplicities.

\subsection{Classifying the contradictory components according to their excesses}
Recall that contradictory circuits and variables are defined
in~\cref{definition:contradictions}.
If a variable \( x \) is contradictory, it can belong to multiple contradictory
circuits at the same time. When the complexity of a contradictory component
increases, the number of circuits simultaneously contaning a given
variable can grow
exponentially on its excess. However, we show that the excess grows very slowly
in the subcritical phase of the 2-SAT transition, so this technique allows to
obtain some structural properties before the number of multiplicities blows up.

Similarly to the excess of a complex component in a simple graph which equals to
the number of its edges minus the number of its vertices, we introduce the
excess of a contradictory graph.
(See~\cref{fig:cscc:excess:one:second,fig:cscc:excess:one,fig:cscc:excess:two,fig:unavoidable:cscc}
for different contradictory components of excess \( 1 \) to \( 2 \).)

\begin{definition}[Contradictory component and its excess]
    \label{definition:contradictory:component}
    We call a digraph \( D \) with \( 2n \) conventionally labelled vertices a
    \emph{contradictory digraph} or a
    \emph{contradictory component} if for its every edge \( e \in D \) its
    complementary edge \( \n e \) is also in \( D \), and every vertex \( x \in
    D \) implies its complementary: \( x \rightsquigarrow \n x \).  The
    \emph{excess} of a contradictory component is defined to be equal to the
    difference between the number of its edges and vertices divided by \( 2 \).
    The excess of an empty graph is defined to be zero.
\end{definition}

\begin{figure}[hbt]
    \begin{minipage}[t]{0.24\textwidth}
\centering
\raisebox{1cm}{
\begin{tikzpicture}[>=stealth',thick, node distance=1.0cm]
\draw
node[arnBleuPetit, minimum height = 1.2em, thick](x)
                    at (0,0)           {\( x \)}
node[arnRougePetit, minimum height = 1.2em, thick](x')
                    at (2.5,0)           {\( \fn x \)}
;
\fromto{x}{x'}     [ultra thick,blackred] [thick]
                    [-30][-5][0.6];
\fromto{x}{x'}     [ultra thick,blackblue][thick]
                    [-75][-5][0.6];
\fromto{x'}{x}     [ultra thick,blackred] [thick]
                    [-30][-5][0.6];
\fromto{x'}{x}     [ultra thick,blackblue][thick]
                    [-75][-5][0.6];
\end{tikzpicture}
}
\caption{
\label{fig:cscc:excess:one:second}
First possible contradictory component of excess \( 1 \).
}
\end{minipage}\hfill
\begin{minipage}[t]{0.24\textwidth}
\centering
\begin{tikzpicture}[>=stealth',thick, node distance=1.0cm]
\draw
node[arnBleuPetit, minimum height = 1.2em, thick](x)
                    at (0,0)           {\( x \)}
node[arnRougePetit, minimum height = 1.2em, thick](x')
                    at (2.5,0)           {\( \fn x \)}
node[arnBleuPetit, minimum height = 1.2em, thick](y)
                    at (2.5,-2.5)           {\( y \)}
node[arnRougePetit, minimum height = 1.2em, thick](y')
                    at (0,-2.5)           {\( \fn y \)}
;
\fromto{x}{x'}     [ultra thick,blackred] [thick]
                    [-50][-5][0.6];
\fromto{x}{x'}     [ultra thick,blackblue][thick]
                    [50][5][0.6];
\fromto{x'}{y}     [ultra thick,blackred] [thick]
                    [-4][-1][0.6];
\fromto{y'}{x}     [ultra thick,blackblue][thick]
                    [-4][-1][0.6];
\fromto{y}{y'}     [ultra thick,blackred] [thick]
                    [-50][-5][0.6];
\fromto{y}{y'}     [ultra thick,blackblue][thick]
                    [50][5][0.6];
\end{tikzpicture}
\caption{
\label{fig:cscc:excess:one}
Second possible contradictory component of excess \( 1 \).
}
\end{minipage}\hfill
\begin{minipage}[t]{0.24\textwidth}
    \centering
\begin{tikzpicture}[>=stealth',thick, node distance=1.0cm]
\draw
node[arnRougePetit, minimum height = 1.2em, thick](y')
                    at (22.5    :  2)           {\( \fn y \)}
node[arnBleuPetit, minimum height = 1.2em, thick](w)
                    at (22.5*3  :  2)           {\( w \)}
node[arnBleuPetit, minimum height = 1.2em, thick](z)
                    at (22.5*5  :  2)           {\( z \)}
node[arnBleuPetit, minimum height = 1.2em, thick](y)
                    at (22.5*7  :  2)           {\( y \)}
node[arnBleuPetit, minimum height = 1.2em, thick](x)
                    at (22.5*9  :  2)           {\( x \)}
node[arnRougePetit, minimum height = 1.2em, thick](z')
                    at (22.5*11 :  2)           {\( \fn z \)}
node[arnRougePetit, minimum height = 1.2em, thick](w')
                    at (22.5*13 :  2)           {\( \fn w \)}
node[arnRougePetit, minimum height = 1.2em, thick](x')
                    at (22.5*15 :  2)           {\( \fn x \)}
;
\fromto{x}{y}      [ultra thick,blackblue][][-4][-1][0.6];
\fromto{y}{z}      [ultra thick,blackblue][][-4][-1][0.6];
\fromto{z}{w}      [ultra thick,blackblue][][-4][-1][0.6];
\fromto{w}{y'}     [ultra thick,blackblue][][-4][-1][0.6];

\fromto{y'}{x'}    [ultra thick,blackred] [][-4][-1][0.6];
\fromto{x'}{z}     [ultra thick,blackred] [][-4][-1][0.6];
\fromto{w}{x}      [ultra thick,blackred] [][-4][-1][0.6];
\fromto{z'}{y'}    [ultra thick,blackred][][4][1][0.6];

\fromto{y}{w'}     [ultra thick,blackred][][4][1][0.6];
\fromto{w'}{z'}    [ultra thick,blackred][][-4][-1][0.6];
\fromto{z'}{x}     [ultra thick,blackblue] [][-4][-1][0.6];
\fromto{x'}{w'}    [ultra thick,blackblue] [][-4][-1][0.6];
\end{tikzpicture}
    \caption{\label{fig:unavoidable:cscc}A minimal contradictory
    component of excess \( \frac{12 - 8}{2} = 2 \).}
    \end{minipage}\hfill
    \begin{minipage}[t]{0.24\textwidth}
        \centering
\raisebox{.2cm}{
\begin{tikzpicture}[>=stealth',thick, node distance=1.0cm]
\draw
node[arnBleuPetit, minimum height = 1.2em, thick](x)
                    at (0,0)           {\( x \)}
node[arnRougePetit, minimum height = 1.2em, thick](x')
                    at (2,0)           {\( \fn x \)}

node[arnBleuPetit, minimum height = 1.2em, thick](z)
                    at (-.1,-1)           {\( z \)}
node[arnRougePetit, minimum height = 1.2em, thick](z')
                    at (2.1,-1)           {\( \fn z \)}

node[arnBleuPetit, minimum height = 1.2em, thick](w)
                    at (1,0.6)           {\( w \)}
node[arnRougePetit, minimum height = 1.2em, thick](w')
                    at (1,-0.6)           {\( \fn w \)}

node[arnBleuPetit, minimum height = 1.2em, thick](y)
                    at (2,-2)           {\( y \)}
node[arnRougePetit, minimum height = 1.2em, thick](y')
                    at (0,-2)           {\( \fn y \)}
;
\fromto{x}{w}      [very thick,blackblue][thick][-20][-5][0.6];
\fromto{w'}{x'}    [very thick,blackred] [thick][-20][-5][0.7];
\fromto{x}{w'}     [very thick,blackred] [thick][50][10][0.7];
\fromto{w}{x'}     [very thick,blackblue][thick][-50][-10][0.6];

\fromto{x'}{z'}    [very thick,blackred] [thick][-4][-1][0.7];
\fromto{z'}{y}     [very thick,blackred] [thick][-4][-1][0.7];
\fromto{y'}{z}     [very thick,blackblue][thick][-4][-1][0.7];
\fromto{z}{x}      [very thick,blackblue][thick][-4][-1][0.7];

\fromto{y}{y'}     [very thick,blackred] [thick][-50][-5][0.6];
\fromto{y}{y'}     [very thick,blackblue][thick][50][5][0.6];

\fromto{z}{w}      [very thick,blackblue]
                   [thick,looseness = 1.6][-100][-5][0.6];
\fromto{w'}{z'}    [very thick,blackred] [thick][20][5][0.6];
\end{tikzpicture}
}
\caption{\label{fig:cscc:excess:two}Contradictory component of excess \( 2 \)
which is not minimal.}
    \end{minipage}
\end{figure}

In contrast with the phase transition of simple graphs, where the whole
structure of the graph is known with high probability, no such information is
available for the phase transition of 2-SAT yet.
On a positive side, Kim has obtained the number of variables and clauses in a
core of a random formula. In this work, we focus on the contradictory component
only.

In order to access the
probability of satisfiability, we use the \emph{inclusion-exclusion method}.
We consider \emph{minimal} contradictory components which are contradictory
implication digraphs that do not have any proper contradictory subgraphs.
Two examples of minimal contradictory components of excess \( 1 \) and \( 2 \)
are given in \cref{fig:cscc:excess:one,fig:unavoidable:cscc}, the arrows play
the role of strictly distinct paths. \cref{fig:cscc:excess:two} represents a
contradictory strongly component of excess \( 2 \) which is not minimal.

\begin{lemma}
    \label{lemma:pie:sat}
Let \( \xi \) denote the number of minimal contradictory components in a random
implication digraph corresponding to a random formula \( F \in \mathcal F(n,m)
\), counted with multiplicities and with possible overlappings.
Then, the probability of satisfiability of a random formula \( F \in \mathcal F(n,m) \)
can be expressed using the principle
of inclusion-exclusion:
\begin{equation}
    \mathbb P(F \in \mathcal F(n,m)\text{ is SAT}) = 1 - \mathbb E_{n,m} \xi
    + \dfrac{1}{2!} \mathbb E_{n,m} \xi(\xi - 1)
    - \dfrac{1}{3!} \mathbb E_{n,m} \xi(\xi - 1)(\xi - 2)
    + \cdots.
\end{equation}
\end{lemma}

\begin{proof}
A formula is satisfiable if and only if
the contradictory component is empty, \ie
the number of minimal contradictory components is
equal to zero.
Let \( N \) denote the total possible number of subgraphs in a digraph with \( n
\) vertices, and let \( E_i \) denote the event that \( i \)th subgraph forms a
minimal contradictory component. Using de~Morgan's rule and the
inclision-exclusion principle, we obtain
\[
    \mathbb P(\xi = 0) =
    \mathbb P(\n{E_1} \land \cdots \land \n{E_N})
    =
    1 - \mathbb P(E_1 \lor \cdots \lor E_N)
    =
    1 - \sum_{i = 1}^N \mathbb P(E_i)
    + \sum_{i<j} \mathbb P(E_i \land E_j)
    - \cdots
    .
\]
Finally, we note that
\[
    \sum_{i_1 < \ldots < i_k}
        \mathbb P(E_{i_1} \land \cdots E_{i_k})
    =
        {N \choose k} \mathbb P(E_1 \land \cdots \land E_k)
    =
        \dfrac{1}{k!} \mathbb E_{n,m}
        \xi(\xi - 1) \cdots (\xi - k+ 1)
        ,
\]
which finishes the proof.
\end{proof}

The term \( \mathbb E \xi(\xi-1) \) denotes the expected number of pairs of
minimal contradictory components, and requires going through different possible
cases of their mutual configuration. Further terms will provide even more
complicated combinatorial structures, but on the asymptotic level, they will all
appear to be negligible in the subcritical phase of the 2-SAT phase transition.

\begin{example}
    \label{example:overlap}
\cref{fig:cscc:excess:two}
    representing a contradictory component of excess \( 2 \) can be also
    considered as a pair of contradictory components each of excess \( 1 \): the
    first one being \( x \rightsquigarrow w \rightsquigarrow \n x
    \rightsquigarrow \n z \rightsquigarrow y \rightsquigarrow \n y
    \rightsquigarrow z \rightsquigarrow x \) with a double sequence \( y
    \rightsquigarrow \n y \) and a mirror path \( x \rightsquigarrow \n w
    \rightsquigarrow \n x \), and the second one being described as
    \( z \rightsquigarrow w \rightsquigarrow \n x \rightsquigarrow \n z
    \rightsquigarrow y \rightsquigarrow \n y \rightsquigarrow z \) with a double
    path \( y \rightsquigarrow \n y \) and a complementary sequence \( z
    \rightsquigarrow x \rightsquigarrow \n w \rightsquigarrow \n z\).
    In such a way, the first component of excess \( 1 \) is obtained by dropping
    the pair \( x \rightsquigarrow w \) and \( \n w \rightsquigarrow \n z \),
    and the second -- by dropping \( x \rightsquigarrow w \) and \( \n w
    \rightsquigarrow \n x \).
\end{example}

\subsection{The case of the simplest minimal contradictory component}
\label{subsection:simplest:cscc}
It is convenient to represent \( \xi \) from~\cref{lemma:pie:sat} as the sum
\( \xi = \sum_{r \geq 1}\xi_r \) where \( \xi_r \) is the number of
distinguished minimal contradictory components of excess \( r \)
in a random implication digraph.

We focus on the case \( r = 1 \) first.
The expected value of \( \xi_1 \) is equal to the number of implication digraphs
with a distinguished contradictory component of excess \( 1 \)
(\eg~\cref{fig:cscc:excess:one}) divided by the
total number of implication digraphs. In order to count the implication digraphs
with distinguished contradictory components, we are going to count their respective
sum-representations and then divide by \( 2^m \) which is the number of
sum-representations of a single implication digraph with \( m \) edges. Both
numerator and denominator of the total fraction are then divided by the same
factor \( 2^m \) which cancels out.
In other words, \( \mathbb E \xi_1 \) can be expressed as
\begin{equation}
    \mathbb E \xi_r = \dfrac{
        \left|\left\{ (G, p) \;\big|\;
            \substack{
                G\text{ is a sum-representation digraph with \( 2n \) vertices
                and \( m \) edges}
            \\
            p\text{ is a minimal contradictory pattern of excess \( r \)
                in \( G + \n G \)
            }
        }\right\}\right|
    }{
        \big|\{
            G \mid G\text{ is a sum-representation digraph with \( 2n \) vertices
                and \( m \) edges}
        \}\big|
    }
    .
\end{equation}
We further note that distinguishing a contradictory pattern in \( G + \n G \)
can be replaced with distinguishing a different contradictory pattern in a
sum-representation digraph \( G \) if a proper multiplicity factor is used.

There are only two possible multigraphs which can result in cancelling
of a contradictory digraph of excess \( 1 \), depicted, respectively,
in~\cref{fig:cscc:excess:one:second,fig:cscc:excess:one}. Accordingly, we
consider two different cases.

\textbf{First case.} Consider a sum-representation digraph with a
distinguished pattern of type depicted in~\cref{fig:first:contradictory:pattern}.
Each arrow there represents a sequence of strictly
distinct literals, and it is also required that all the literals on
both sequences (except \( x \) and \( \n x \)) are pairwise strictly distinct, and
that the two distinguished literals marked \( x \) and \( \n x \) are
complementary.  Suppose that the pattern has \( \ell \) oriented edges. Then,
among \( 2^\ell \) graphs obtained by edge rotations of the pattern, there are
exactly \( 4 \) which belong to the same pattern: it is possible to take a
complementary of the path \( x \rightsquigarrow \n x \) which gives \( 2 \)
possibilities, and the complementary of the path \( \n x \rightsquigarrow x \).

\textbf{Second case.} Consider a sum-representation digraph with a
distinguished pattern of type depicted in~\cref{fig:second:contradictory:pattern}.
Again, each arrow corresponds to a sequence, and all the literals on the
three sequeces \( x \rightsquigarrow \n x \), \( \n x \rightsquigarrow y \), \(
y \rightsquigarrow \n y \) are pairwise strictly distinct, except for \( x \)
and \( \n x \), and \( y \) and \( \n y \), which are required to be
complementary. The literals \( x \) and \( y \) are also required to be strictly
distinct. Again, if the total length of the pattern is \( \ell \), there are \(
8 \) sum-representations among \( 2^\ell \) equivalent ones, such
that these sum-representations form the same pattern. Indeed, \( 4 \)
options are given by replacing either of the two paths \( x \rightsquigarrow \n
x\) and \( y \rightsquigarrow \n y \) by its complementary. If the path \( \n x
\rightsquigarrow y \) is replaced by its complementary, then we obtain a
different graph constructed of the three sequences \( x \rightsquigarrow \n x
\), \( \n y \rightsquigarrow x \), \( y \rightsquigarrow \n y \) which belongs
to the same pattern if we swap the variables \( x \) and \( y \). No other
combinations of edge rotations give the same pattern because it is required that
all the literals are pairwise strictly distinct.

\begin{figure}[hbt]
    \begin{minipage}[t]{0.45\textwidth}
    \centering
\begin{tikzpicture}[>=stealth',thick, node distance=1.0cm]
\draw
node[arnBleuPetit, minimum height = 1.2em, thick](x)
                    at (0,0)           {\( x \)}
node[arnRougePetit, minimum height = 1.2em, thick](x')
                    at (2,0)           {\( \fn x \)}
;
\fromto{x}{x'}     [very thick,blackred][thick][-84][-1][0.6];
\fromto{x'}{x}     [very thick,blackred][thick][-84][-1][0.6];
\end{tikzpicture}
\caption{\label{fig:first:contradictory:pattern}First sum-representation
contradictory pattern corresponding to \cref{fig:cscc:excess:one:second}.}
\end{minipage}\hfill
\begin{minipage}[t]{0.53\textwidth}
    \centering
\raisebox{.7cm}{
\begin{tikzpicture}[>=stealth',thick, node distance=1.0cm]
\draw
node[arnBleuPetit, minimum height = 1.2em,  thick](x)
                    at (0,0) {\( x \)}
node[arnRougePetit, minimum height = 1.2em, thick](x')
                    at (2,0) {\( \fn x \)}
node[arnBleuPetit, minimum height = 1.2em,  thick](y)
                    at (4,0) {\( y \)}
node[arnRougePetit, minimum height = 1.2em, thick](y')
                    at (6,0) {\( \fn y \)}
;
\fromto{x}{x'}     [very thick,blackred][thick][4][1][0.6];
\fromto{x'}{y}     [very thick,blackred][thick][-4][-1][0.6];
\fromto{y}{y'}     [very thick,blackred][thick][4][1][0.6];
\end{tikzpicture}
}
\caption{\label{fig:second:contradictory:pattern}Second sum-representation
contradictory pattern corresponding to \cref{fig:cscc:excess:one}.}
\end{minipage}
\end{figure}

Suppose that a pattern (either from~\cref{fig:first:contradictory:pattern}
or from~\cref{fig:second:contradictory:pattern}) has \( \ell \) edges and is a
distinguished subgraph of a sum-representation digraph containing \( m \) edges
in total. Then, among \( 2^m \) equivalent, \( 2^{m-\ell} \)
contain the pattern unaltered. Therefore, the number of
sum-representation digraphs with a distinguished contradictory pattern of length
\( \ell \)
from~\cref{fig:first:contradictory:pattern,fig:second:contradictory:pattern}
counted with multiplicity \( 2^\ell \) enumerates (with respective
multiplicities 1-to-4 and 1-to-8) implication digraphs with a distinguished
contradictory digraph~\cref{fig:cscc:excess:one:second,fig:cscc:excess:one} of
excess 1.

Combining the two cases, we obtain a new expression for \( \mathbb E \xi_1 \):
\begin{align*}
    \mathbb E \xi_1 &=
    \dfrac14 \cdot
    \dfrac{
        \sum_{\ell = 0}^\infty
        2^\ell
        \left|\left\{
            (G, p_\ell) \ | \
            \substack{
                G\text{ is a sum-representation digraph with \( 2n \) vertices
                and \( m \) edges}
            \\
            p_\ell \text{ is a pattern from~\cref{fig:first:contradictory:pattern}
                of length \( \ell \) in \( G \)
            }}
        \right\}\right|
    }{
        \big|\{
            G \mid G\text{ is a sum-representation digraph with \( 2n \) vertices
                and \( m \) edges}
        \}\big|
    }
    \\
    & +
    \dfrac18 \cdot
    \dfrac{
        \sum_{\ell = 0}^\infty
        2^\ell
        \left|\left\{
            (G, p_\ell) \ | \
            \substack{
                G\text{ is a sum-representation digraph with \( 2n \) vertices
                and \( m \) edges}
            \\
            p_\ell \text{ is a pattern from~\cref{fig:second:contradictory:pattern}
                of length \( \ell \) in \( G \)
            }}
        \right\}\right|
    }{
        \big|\{
            G \mid G\text{ is a sum-representation digraph with \( 2n \) vertices
                and \( m \) edges}
        \}\big|
    }
    .
\end{align*}
As we shall see in~\cref{section:asymptotic:analysis}, the
contribution of the first summand is negligible. An intuitive way to see why
this happens is the following: the shape from~\cref{fig:cscc:excess:one:second}
can be obtained by contracting a path \( \n y \rightsquigarrow x \) in
\cref{fig:cscc:excess:one}. Therefore, the first case is a degenerate case when
the length of the corresponding path is equal to zero and the literal \( x \)
coincides with the literal \( \n y \). In general, only the structures with
cubic kernels have dominant contribution.

A similar decomposition can be obtained for \( \mathbb E \xi_r \) for any finite \(
r \).  Instead the two presented cases, the sum should be taken over all the
possible minimal contradictory components of given excess \( r \), divided by a
corresponding multiplicity factor, which plays the analog of a
compensation factor for simple graphs. A detailed explanation of the nature of
such compensation factors is given in~\cref{subsection:compensation:factors}.

\section{Generating functions and saddle point techniques}
\label{section:generating:functions}
The symbolic method and analytic combinatorics~\cite{flajolet2009analytic}
are powerful methods which allow to
\begin{enumerate*}[label=(\textit{\roman*})]
    \item construct combinatorial objects (like graphs, or 2-SAT formulae, in
        our case) by combining them in various different ways using a set of
        allowed operations;
    \item translate this constructions into the language of generating
        functions;
    \item analyse the asymptotic number of objects using
        various available integration tools.
\end{enumerate*}
This section is devoted to describing the symbolic construction of graphs and
directed
graphs needed for our purpose, and the asymptotic tool needed for our analysis.

\subsection{Symbolic method for graphs}
Recall that
the \emph{size} of a graph (or a directed graph) is defined to be equal to the
number of its vertices.
In this paper, labelled graphs are considered: every vertex has a distinct label
from the set \( \{1, 2, \ldots, n\} \) where \( n \) is the total number of
vertices of the graph.

Given a family of graphs (or directed graphs) \( \mathcal A \), we construct its
corresponding \emph{exponential generating function} (EGF) which is
\[
    A(z) := \sum_{n \geq 0} a_n \dfrac{z^n}{n!},
\]
where \( a_n \) is equal to the number of graphs or digraphs of size \( n \)
from \( \mathcal A \).
The operator \( [z^n] \) of taking \( n \)th coefficient is then defined as
\( [z^n] A(z) = a_n / n! \).
All the indefinite integrals \( \int F(z) dz \) in this paper should be
interpreted as \( \int_0^z F(\tau) d \tau \).

Given two families \( \mathcal A \) and \( \mathcal B \),
the sum of their respective generating functions \( A(z) \) and \( B(z) \)
represents the union of the two families \( \mathcal A \) and \( \mathcal B \).
The family \( \fC := \fA \times \fB \) is defined as the family of labelled
graphs (or digraphs) whose vertices are partitioned into two sets such that
there are no edges between the parts, the underlying graph of the first part is
from \( \mathcal A \), and the graph from the second part is from \( \mathcal B
\). We can also say that \( \fA \times \fB \) is the labelled family of ordered
tuples of graphs from \( \fA \) and \( \fB \). If the respective EGFs of these
families are \( \eA(z) \) and \( \eB(z) \) then the EFG of \( \fC \) is
\( A(z) \cdot B(z) \).

Consequently, if \( A(z) \) is an EGF for a graph family \( \mathcal A \) not
containing an empty graph, then \( A^k(z) \) is the EGF for sequences of length
\( k \) of graphs from \( \mathcal A \). Summing over all \( k \), we obtain the
EGF for all possible sequences (possibly empty) of graphs from \( \mathcal A \),
which is \( \frac{1}{1 - A(z)} \).
If the family \( \mathcal A \) does not contain an empty graph, then
the EGF for non-ordered sequences of length \( k \) (\ie sets) of graphs from
\( \mathcal A \) is \( A(z)^k/k! \). Summing over all \( k \), we
obtain the EGF of all unordered sequences of graphs from \( \mathcal A \), which
is \( e^{A(z)} \). Finally, a oriented cyclic composition of graphs from \(
\mathcal A \) has EGF \( \sum_{k \geq 1} A^k(z) / k = \log \frac{1}{1 - A(z)} \).
If a cycle is not oriented, then the corresponding EGF is \( \frac12 \log
\frac{1}{1 - A(z)} \).

\begin{figure}[hbt]
    \begin{minipage}[t]{0.31\textwidth}
\centering
\raisebox{1.0cm}{
\begin{tikzpicture}[>=stealth',thick, node distance=1.0cm]
\draw
node[arnBleuPetit](a) at ( 0, 0)  { }
node[arnBleuPetit](s) at ( 1,-1)  { }
node[arnBleuPetit](d) at ( 2, 0)  { }
node[arnBleuPetit](f) at ( 1, 1)  { }
;
\sedge{a}{d}[][4];
\sedge{a}{f}[][4];
\sedge{s}{a}[][4];
\sedge{d}{s}[][4];
\sedge{f}{d}[][4];
\draw
node[arnRougePetit](z)  at (3,   1)  { }
node[arnRougePetit](x)  at (3,   0)  { }
node[arnRougePetit](q)  at (3.8, 0)  { }
node[arnRougePetit](w)  at (3,  -1)  { }
node[arnRougePetit](e)  at (4.8,-1)  { }
node[arnRougePetit](r)  at (4.8, 0)  { }
node[arnRougePetit](t)  at (3.8, 1)  { }
;
\sedge{q}{w}[][4];
\sedge{q}{r}[][4];
\sedge{e}{r}[][4];
\sedge{r}{t}[][4];
\sedge{t}{q}[][4];
\sedge{z}{x}[][4];
\node[rectangle,dashed,draw,fit=(a)(s)(d)(f), very thick,
      rounded corners=5mm,inner sep= 5pt, bgreen] {};
\node[rectangle,dashed,draw,fit=(q)(w)(e)(r)(t), very thick,
      rounded corners=5mm,inner sep= 5pt, bgreen] {};
\end{tikzpicture}
}
\caption{Product of graph families.}
\end{minipage}\hfill
\begin{minipage}[t]{0.31\textwidth}
    \centering
\begin{tikzpicture}[>=stealth',thick, node distance=1.0cm]
\draw
node[arnRougePetit](a) at (360/5*0 : 1.2)  {}
    node[arnBleuPetit](a1) at ($(a)+(1,0)$) {}
        node[arnBleuPetit](a11) at ($(a1)+(1,0.4)$) {}
        node[arnBleuPetit](a12) at ($(a1)+(1,-0.4)$) {}
node[arnRougePetit](s) at (360/5*1 : 1.2)  { }
    node[arnBleuPetit](s1) at ($(s)+(72+20 : 1)$)  { }
    node[arnBleuPetit](s2) at ($(s)+(72-20 : 1)$)  { }
node[arnRougePetit](d) at (360/5*2 : 1.2)  { }
    node[arnBleuPetit](d1) at ($(d)+(360/5*2 : 1)$) {}
node[arnRougePetit](f) at (360/5*3 : 1.2)  { }
node[arnRougePetit](g) at (360/5*4 : 1.2)  { }
    node[arnBleuPetit](g1) at ($(g)+(-72+30 : 1)$)  { }
    node[arnBleuPetit](g2) at ($(g)+(-72 : 1)$)  { }
    node[arnBleuPetit](g3) at ($(g)+(-72-30 : 1)$)  { }
;
\sedge{a}{s}   [][16];
\sedge{s}{d}   [][16];
\sedge{d}{f}   [][16];
\sedge{f}{g}   [][16];
\sedge{g}{a}   [][16];
\sedge{a}{a1}  [][4];
\sedge{a1}{a11}[][4];
\sedge{a1}{a12}[][4];
\sedge{s}{s1}  [][4];
\sedge{s}{s2}  [][4];
\sedge{d}{d1}  [][4];
\sedge{g}{g1}  [][4];
\sedge{g}{g2}  [][4];
\sedge{g}{g3}  [][4];
\node[dashed,draw,fit=(a)(a1)(a11)(a12), very thick,
      rounded corners=5mm,inner sep= 5pt, bgreen] {};
\node[dashed,draw,rotate fit = 72, fit=(s)(s1)(s2), very thick,
      rounded corners=5mm,inner sep= 5pt, bgreen] {};
\node[dashed,draw,rotate fit = 360/5*2, fit=(d)(d1), very thick,
    rounded corners = 7pt,inner sep= 5pt, bgreen] {};
\node[dashed,draw,rotate fit = 360/5*3, fit=(f), very thick,
    rounded corners = 7pt,inner sep= 5pt, bgreen] {};
    \node[dashed,draw,rotate fit = -72, fit=(g)(g1)(g2)(g3), very thick,
    rounded corners = 5mm,inner sep= 4pt, bgreen] {};
\end{tikzpicture}
\caption{Cyclic composition.}
\end{minipage}\hfill
\begin{minipage}[t]{0.31\textwidth}
    \centering
    \raisebox{1.0cm}{
\begin{tikzpicture}[>=stealth',thick, node distance=1.0cm]
\draw
node[arnRougePetit, minimum height = 1.1em, ultra thick](R)
                           at (0, 0)                        {}
node[arnBleuPetit](a)      at ($(R)-(1.2,0.5)$)             {}
node[arnBleuPetit](a1)     at ($(a)+(-.5,-0.5)$)            {}
node[arnBleuPetit, minimum height = 1.1em, ultra thick](a11)
                    at ($(a1)+(0.5,-0.5)$)           {\( 1 \)}
node[arnBleuPetit](a111)   at ($(a11)-(0.2,0)+(0.5, -0.5)$) {}
node[arnBleuPetit](a112)   at ($(a11)-(0.2,0)+(0,   -0.5)$) {}
node[arnBleuPetit](a113)   at ($(a11)-(0.2,0)+(-0.5,-0.5)$) {}
node[arnBleuPetit](a12)    at ($(a1)+(-.5,-0.5)$)           {}
node[arnBleuPetit](R1)     at ($(R)  +  (0,-0.5)$)          {}
node[arnBleuPetit](R11)    at ($(R1) +  (0,-0.5)$)          {}
node[arnBleuPetit](R2)     at ($(R)  +  (0.5,-0.5)$)        {}
node[arnBleuPetit](R21)    at ($(R2) +  (0,-0.5)$)          {}
node[arnBleuPetit](R211)   at ($(R21)+  (-.5,-0.5)$)        {}
node[arnBleuPetit](R2111)   at ($(R211)+ (0,-0.5)$)         {}
node[arnBleuPetit](R2112)   at ($(R211)+ (.5,-0.5)$)        {}
node[arnBleuPetit](R2113)   at ($(R211)+ ( 1,-0.5)$)        {}
node[arnBleuPetit](R212)   at ($(R21)+  (.5,-0.5)$)         {}
;
\node[dashed,draw,fit=(R)(R2111)(R2113), very thick,
      rounded corners=5mm,inner sep= 5pt, bgreen] {};
\node[dashed,draw,fit=(a)(a12)(a11)(a111)(a113), very thick,
      rounded corners=5mm,inner sep= 5pt, bgreen] {};
\sedge{R}{a}      [ultra thick][4];
\sedge{a}{a1}     [ultra thick][4];
\sedge{a1}{a11}   [ultra thick][4];
\sedge{a11}{a111} [][4];
\sedge{a11}{a112} [][4];
\sedge{a11}{a113} [][4];
\sedge{a1}{a12}   [][4];
\sedge{R}{R1}     [][4];
\sedge{R}{R2}     [][4];
\sedge{R1}{R11}   [][4];
\sedge{R2}{R21}   [][4];
\sedge{R21}{R211} [][4];
\sedge{R21}{R212} [][4];
\sedge{R211}{R2111}[][4];
\sedge{R211}{R2112}[][-4];
\sedge{R211}{R2113}[][-4];
\end{tikzpicture}
}
\caption{Constructing unrooted trees: the case when the label of the root is not
equal to \( 1 \).}
\end{minipage}
\end{figure}

If all the graphs in a family \( \mathcal A \) are connected,
then \( e^{A(z)} \) enumerates labelled graphs whose connected components are from \(
\mathcal A \). The product of exponential generating functions for labelled
graphs can be given further interpretations: if two graph families \( \mathcal A
\) and \( \mathcal B \) do not intersect and the graphs from these families are
connected, then the product of their EGFs enumerates graphs with two connected
components, one from \( \mathcal A \) and the second from \( \mathcal B \).
In the same manner, one can multiply generating functions of not necessarily
connected graph families provided that the underlying connected components
corresponding to the two families, are always distinct. Then, the product can be
interpreted as a new family of graphs whose vertices can be partitioned uniquely
into two sets, and the graphs constructed on the respective sets, belong to the
first and the second family, respectively.

It is well known that the EGF \( T(z) \) for rooted trees,
also known as Cayley trees, satisfies
\[
    T(z) = z e^{T(z)} = \sum_{n \geq 0} n^{n-1} \dfrac{z^n}{n!}.
\]
\emph{Unicyclic graphs} are defined as connected graphs whose number of vertices
is equal to their number of edges. Every such graph has exactly one undirected
cycle inside it and can be represented as a sequence of trees arranged in a
cycle of length at least \( 3 \), which results in the following EGF \( V(z) \):
\[
    V(z) =
    \dfrac12 \sum_{k \geq 3} \dfrac{T(z)^k}{2k}
    =
    \dfrac{1}{2}\left[
        \log \dfrac{1}{1 - T(z)} - T(z) - \dfrac{T^2(z)}{2}
    \right].
\]

The EGF for \emph{unrooted trees} has the form \( U(z) = T(z) - \frac{T(z)^2}{2}
\). An elegant proof of this fact is given in~\cite{flajolet1989first}.
We present a sketch of the proof for completeness.
The label of the root of a rooted tree can be equal to \( 1 \) or
greater than \( 1 \). The generating function of rooted trees whose root has
label \( 1 \) is \( U(z) \) because an unrooted tree can be canonically rooted
at vertex with label \( 1 \). Otherwise, a tree whose root label is greater than
\( 1 \) can be represented as a pair consisting of the subtree whose parent is
the root and which contains the vertex with label \( 1 \), and the remaining
tree which is formed by removing the aforementioned subtree from the initial
tree. The generating function of such (unordered) tuples is \( T^2(z) / 2 \).
Adding up the two cases, we obtain an identity \( T(z) = U(z) + T^2(z) / 2 \).

Similarly to trees and unicycles, more complex structures can be defined. An
\emph{excess} of a connected graph is equal to the number of its edges minus the
number of its vertices. A connected graph with minimal possible excess is a
tree, and its excess is equal to \( -1 \). Next, a unicycle is a connected graph
of excess \( 0 \). A connected graph of excess \( 1 \) is called a
\emph{bicycle}.

Suppose that a connected graph \( G \) has excess \( r \). The \emph{pruning}
procedure is described as repeated removal of vertices of degree \( 1 \) until
no vertices of degree \( 1 \) remain. The resulting graph is then called the
\emph{2-core} or just the \emph{core} of \( G \).
Consequently, each vertex of degree \( 2 \) of \( G \) can be removed, while
connecting its former neighbours by an edge; this procedure is called
\emph{cancelling}. The resulting graph obtained by pruning and cancelling is
called the \emph{3-core}, or the \emph{kernel} of \( G \).
After cancelling, a simple graph may gain some multiple edges
and loops, \ie become a multigraph. Note that pruning and cancelling do not
change the excess of \( G \). It is known (see \eg~\cite{janson1993birth}) that
after pruning and cancelling, the resulting multigraph belongs to a finite set
of multigraphs with given excess \( r \), see
also~\cref{subsection:compensation:factors}.

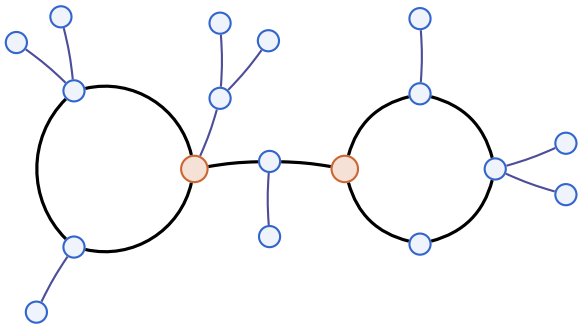
\begin{figure}[hbt]
\centering
\begin{tikzpicture}[>=stealth',thick, node distance=1.0cm]
\draw
node[arnRougePetit, minimum height = 1.0em](q) at (360/3*0 : 1.2)  {}
    node[arnBleuPetit](q01) at ($(q)+(70:1)$) {}
        node[arnBleuPetit](q011) at ($(q01)+(70-20:1)$) {}
        node[arnBleuPetit](q012) at ($(q01)+(70+20:1)$) {}
node(q0)               at ($(q)-(1,0)$)  {}
node[arnBleuPetit](q1) at ($(q0)+(120:1.2)$)  {}
    node[arnBleuPetit](q11) at ($(q1)+(120-20:1)$) {}
    node[arnBleuPetit](q12) at ($(q1)+(120+20:1)$) {}
node[arnBleuPetit](q2) at ($(q0)+(240:1.2)$)  {}
    node[arnBleuPetit](q21) at ($(q2)+(240:1)$) {}
node[arnBleuPetit](qw) at ($(q)+(1,0.1)$)  {}
    node[arnBleuPetit](qw1)at ($(qw)+(0,-1)$)  {}
node[arnRougePetit, minimum height = 1.0em](w) at ($(q)+(2,0)$)  {}
node(w0)               at ($(w)+(1,0)$)  {}
node[arnBleuPetit](w1) at ($(w0)+(90:1)$)  {}
    node[arnBleuPetit](w11) at ($(w1)+(90:1)$)  {}
node[arnBleuPetit](w2) at ($(w0)+(0:1)$)  {}
    node[arnBleuPetit](w21) at ($(w2)+(-20:1)$)  {}
    node[arnBleuPetit](w22) at ($(w2)+(+20:1)$)  {}
node[arnBleuPetit](w3) at ($(w0)+(-90:1)$)  {}
;
\sedge{q}{q1}   [very thick][45];
\sedge{q1}{q2}  [very thick][45];
\sedge{q2}{q}   [very thick][45];
\sedge{q}{qw}   [very thick][-4];
\sedge{qw}{w}   [very thick][-4];
\sedge{w}{w1}   [very thick][-30];
\sedge{w1}{w2}  [very thick][-30];
\sedge{w2}{w3}  [very thick][-30];
\sedge{w3}{w}   [very thick][-30];

\sedge{q}{q01}    [][4];
\sedge{q01}{q011} [][4];
\sedge{q01}{q012} [][4];
\sedge{q1}{q11}   [][4];
\sedge{q1}{q12}   [][4];
\sedge{q2}{q21}   [][4];
\sedge{qw}{qw1}   [][4];
\sedge{w1}{w11}   [][4];
\sedge{w2}{w21}   [][4];
\sedge{w2}{w22}   [][4];
\end{tikzpicture}
\caption{Pruning and cancelling.}
\end{figure}

It turns out that in the subcritical and in the critical phases, the dominant
contribution comes only from the cubic (3-regular) kernels. If a kernel is not
cubic, it can be obtained as a ``limiting case'' from some cubic kernel by
contracting some of its edges.
This corresponds to the case when the corresponding paths in the initial graph
have length zero. Heuristically, since the average length of a path on the
cubic core is of order \( \Theta(n^{1/3}) \) in the critical phase (see
\eg~\cite{janson2011random}), the probability that a path has zero vertices, \ie
that the kernel is non-cubic, is of order \( O(n^{-1/3}) \).

Turning to the case of digraphs and implication digraphs,
pruning is not defined for a contradictory component, because it doesn't contain
any
vertices of degree \( 1 \). Cancellation is then defined in the following way: if a
vertex has in-degree \( 1 \) and out-degree \( 1 \), it is removed and a
directed edge between the in-neighbour and the out-neighbour is added.
It is easy to see that a directed multigraph obtained after a cancellation
procedure with given excess belongs to a finite sets of contradictory
reduced multigraphs of given excess (see
also~\cref{problem:enumeration:of:cscc}).
If in an implication multidigraph every vertex has the sum of in- and
out-degrees equal \( 3 \), it is also called cubic.

\subsection{Weakly connected directed graphs}
\label{subsection:weakly:connected:digraphs}

A \emph{weakly connected} digraph is a directed graph whose underlying
non-oriented projection is connected.  We construct the EGFs of the analogs of the
trees and unicycles in the world of directed graphs.

A rooted tree of size \( n \) has \( n-1 \) edges, and for each of the edges
there are two orientation choices. For unrooted trees, we can take the vertex
with label \( 1 \) as a canonical root, so that all of the \( 2^{n-1} \) edge
orientations can be also distinguished and result in distinct oriented unrooted
trees.
Therefore, the EGFs for, respectively, oriented rooted and unrooted trees
(\ie directed weakly connected graphs whose non-oriented projections are,
respectively, rooted and unrooted trees) are, respectively,
\begin{equation}
    T_{\to}(z) = \dfrac{1}{2} T(2 z), \quad
    \text{and} \quad
    U_{\to}(z) = \dfrac{1}{2} U(2 z)
    =
    \dfrac12 T(2z) - \dfrac14 T^2(2z)
    .
\end{equation}
The EGF for simple digraphs whose underlying non-oriented projections are
cycles of length at least \( 3 \), is obtained by substitution \( z \mapsto 2z
\) into that of cyclic non-oriented graphs, and equals
\( \dfrac12\left[
    \log \dfrac{1}{1 - 2z} - (2z) - \dfrac{(2z)^2}{2}
\right]\). In simple digraphs, it is allowed to have a circuit of length \( 2 \)
on condition that it has the form \( x \to \n x \to x\), \ie connects two
vertices in both directions. Adding this case corresponds to the summand
\( z^2/2 \).

By substitution, it follows that the EGF for unicyclic directed graphs
(directed graphs whose non-oriented projections are unicyclic graphs) is
\begin{equation}
    V_{\to}(z) =
    \dfrac12\left[
        \log \dfrac{1}{1 - 2T_{\to}(z)} - (2T_\to(z)) - \dfrac{(2T_{\to}(z))^2}{2}
    \right]
    + \dfrac{T_\to(z)^2}{2}
    .
\end{equation}
The last summand is taken out intentionally.
Essentially, \( V_{\to}(z) = V(2z) + T^2(2z) / 8 \), where \( V(z) \) is the
previously obtained EGF for unicyclic simple graphs.
EGFs for directed graphs with higher excess can be constructed in the same way.

\begin{remark}
    Apart from weakly connected digraphs, which come as a more or less direct
    generalisation of the corresponding non-oriented versions, it is also
    possible to specify digraphs with respect to their \emph{strongly connected
    components}, see~\cite{de2019symbolic} and references therein.
    This may lead to further development of more
    precise properties of the 2-SAT phase transition, or that of critical random
    digraphs.
\end{remark}

\subsection{Compensation factors}
\label{subsection:compensation:factors}
The compensation factor of a multigraph \( M \) is a certain coefficient
depending on \( M \) which is required to obtain the EGF of the family of graphs
reducing to \( M \) under pruning and cancelling. While for the purposes of the
current paper we take a relatively low-level classical approach, there also
exists an elegant unifying framework~\cite{panafieu2016analytic} for dealing
with compensation factors by introducing bivariate generating functions which
are exponential with respect to two variables.

Consider the three possible kernels of excess \( 1 \)
in~\cref{figure:bicyclic:kernels}. Each of these kernels a certain symmetry
group which acts on the set of its half-edges and vertices. Since the
objects are labelled, the cardinality of the automorphism group acting on its
vertices equals \( n! \) divided by the number of distinct labelled graphs.
The remaining compensation factor can be determinted solely on the base of the
incidence matrix.

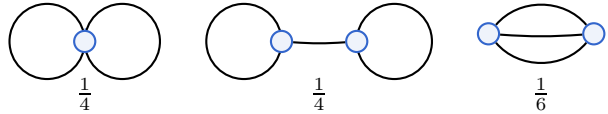
\begin{figure}[hbt]
\centering
\begin{tikzpicture}[>=stealth',thick, node distance=1.0cm]
\draw
node[arnBleuPetit](a) at (0.0, 0.0) {}
node at (0.0, -.7) {\(\tfrac14\)}
[thick] (a) arc [radius = 0.5, start angle= 0, end angle= 360]
[thick] (1.0, 0.0) arc [radius = 0.5, start angle= 0, end angle= 360]
;
\path (a) edge [bblue, very thick, bend left=90] (a) ;
\end{tikzpicture}
\( \quad \)
\begin{tikzpicture}[>=stealth',thick, node distance=1.0cm]
\draw
node[arnBleuPetit](a) at (0.0, 0.0) {}
node[arnBleuPetit](b) at (1.0, 0.0) {}
node at (0.5, -.7) {\(\tfrac14\)}
[thick] (a) arc [radius = 0.5, start angle = 0, end angle= 360]
[thick] (2.0, 0.0) arc [radius = 0.5, start angle = 0, end angle= 360]
;
\sedge{a}{b}[ ][4];
\end{tikzpicture}
\( \quad \)
\begin{tikzpicture}[>=stealth',thick, node distance=1.0cm]
\draw
node[arnBleuPetit](a) at (0.0, 0.0) {}
node[arnBleuPetit](b) at (1.4, 0.0) {}
node at (0.7, -.8) {\(\tfrac16\)}
;
\sedge{a}{b}[ ][4];
\path (a) edge [thick, bend left = 50] (b) ;
\path (a) edge [thick, bend right = 50] (b) ;
\end{tikzpicture}
\caption{
    \label{figure:bicyclic:kernels}
    Kernels of complex components of excess \( 1 \) and their respective
compensation factors.}
\end{figure}

\begin{definition}
    Consider a multigraph \( M \) with \( n \) labelled vertices and with \(
    m_{xy} \) edges between vertices with labels \( x \) and \( y \). In
    particular, \( m_{xx} \) denotes the number of loops of a vertex \( x \),
    and \( m_{xy} = m_{yx} \).

    A \emph{compensation factor} \( \varkappa(M) \) is defined as follows:
    \begin{equation}
        \dfrac{1}{\varkappa(M)}
        :=
        \prod_{x = 1}^n 2^{m_{xx}} \prod_{y=x}^n m_{xy}!
        .
    \end{equation}
\end{definition}
The number of ways to construct a single multigraph \( M \) with \( n \)
vertices and \( m \) edges by gluing the half-edges of the labelled vertices can
be viewed as \( 2^m m! \varkappa(M) \).
The EGF of multigraphs \( G \) reducing under pruning and cancelling to a
given multigraph \( M \) is then written as
\[
    \dfrac{\varkappa(M)}{n!} \cdot \dfrac{T(z)^n}{(1 - T(z))^m}
    ,
\]
which can be obtained by substitution of trees into the corresponding sequences
and vertices. For a rigorous proof of this expression,
see~\cite{janson1993birth}.

For contradictory components in implication digraphs, we define
the compensation factors in a very similar way.

\begin{definition}[Compensation factor of sum-representations and implication
    digraphs]
    Consider a sum-representation digraph \( D \in \mathcal D^\circ(2n,m) \) and
    a contradictory implication digraph \( M \) with \( 2n \)
    conventionally labelled vertices and \( 2m \) directed edges.
    Suppose that for each literals \( x \) and \( y \) there are
    \( d_{xy} \) oriented edges in \( D \) and
    \( m_{xy} \) oriented edges in \( M \)
    between vertices with labels
    \( x \) and \( y \). The compensation factors of \( D \) and \( M \)
    are then defined as
    \begin{equation}
        \dfrac{1}{\varkappa(D)} :=
        \prod_{x = 1}^n \prod_{y = 1}^n d_{xy}!
        , \quad
        \dfrac{1}{\varkappa(M)} :=
        \prod_{x = 1}^n \prod_{y = 1}^n m_{xy}!!
        ,
    \end{equation}
    where for even \( N \), the double factorial is defined as \( N!! = 2 \cdot
    4 \cdot \cdots \cdot N \).
\end{definition}
For the case of simple graphs, the sum over all possible labelled kernels \( M
\) with \( n \) vertices and given excess \( r \),
\(
    \sum_{M} \frac{\varkappa(M)}{n!}
\) expresses the coefficient at \( |\mu|^{-3r} \) in the probability that a
random graph has
a complex component of such excess
(c.f.\xspace the asymptotic probability \( 5 / 24 \cdot |\mu|^{-3} \) of having a
bicyclic complex component, where \( 5/24 \) equals the sum of the compensation
factors \( 1/4 \) and \( 1/6 \) divided by \( 2! \),
see also~\cref{remark:bicyclic:components}).

As we show in~\cref{theorem:satisfiability}, the compensation factor for
contradictory components plays exactly the same role.
The compensation factor for contradictory components has the following additional
interpretation.

\begin{lemma}
    \label{lemma:compensation:factor:interpretation}
    Consider a contradictory component \( \mathcal C \) with \( n \) Boolean
    variables and one of its
    sum-representations \( \pi(\mathcal C) \). We shall say that a
    sum-representation digraph \( \pi_1 \) is \emph{isomorphic} to a digraph
    \( \pi_2 \) if \( \pi_1 \) can be obtained from \( \pi_2 \) by a permutation
    of Boolean variables.
    Then, the number of sum-representations, both equivalent and isomorphic to
    \( \pi(\mathcal C) \) is equal to
    \( n! \varkappa(\pi) / \varkappa(\mathcal C) \).
\end{lemma}
\begin{proof}
    If \( \mathcal C_1, \ldots, \mathcal C_K \) are \( K \) possible isomorphic
    implication
    digraphs obtained by label permutations, then each of the digraphs has
    the same number of
    \( 2^m \) sum-representations, and therefore, each isomorphic
    sum-representation is counted with multiplicity \( K \).

    Let us compute the number \( K \) of possible isomorphic digraphs obtained
    by label permutations. All the edges of \( \mathcal C \) come in pairs,
    therefore, within each pair there is a cyclic group of order 2, and there
    are \( (m_{xy}/2)! \) permutations between the \( (m_{xy}/2) \) pairs of
    oriented edges between the vertices \( x \) and \( y \). Multiplying the
    orders of these groups, we obtain \( (m_{xy}/2)! \cdot 2^{m_{xy}/2} =
    (m_{xy})!!  \). Each loop is directed, so there is no additional factors
    corresponding to the loops. The number \( K \) then equal to \( n! /
    \varkappa(\mathcal C) \) since this is equal to the cardinality of the
    automorphism group.
\end{proof}

\begin{example}
As an illustration of this concept, let us compute the compensation factor of
the contradictory component of excess \( 1 \)
from~\cref{fig:cscc:excess:one}. There are \( n = 2 \) Boolean variables
and two multiple edges, between, respectively, \( x \) and \( \n x \), and \( y
\) and \( \n y \). Therefore, the compensation factor is \( 1/2!!^2 = 1/4 \). The
number of isomorphic equivalent sum-representations is therefore \( 2! \cdot 4 =
8 \), which explains the factor \( 1/8 \) in the second summand of the
  expression for \( \mathbb \xi_1 \) in~\cref{subsection:simplest:cscc}.
  The factor \( 1/4 \) corresponding to~\cref{fig:cscc:excess:one:second} is
  computed as \( 1! \cdot 2!!^2 \) because there is one Boolean variable and two
  double edges.
\end{example}

It is natural to expect that the concept of compensation factor of a
contradictory component will prove helpful not only for the subcritical phase of
the phase transition, but will allow to give a complete description of the
transition curve.
See also~\cref{section:open:problems} for discussions and open questions.

\subsection{Incorporating relations between labels of vertices}

We continue to use the labelling convention for digraphs and implication
digraphs introduced in~\cref{section:sum:representation}, so that for digraphs
having \( 2n \) vertices, the labels of the vertices are partitioned into the
set of positive literals \( \{1, \cdots, n\} \) and negative ones
\( \{\n 1, \cdots, \n n \} \).

\begin{definition}[Literal linking construction]
\label{definition:literal:linking:construction}
Suppose that in a given family \( \fA \) of digraphs, each graph \( G \in
\fA \) contains two distinguished \emph{empty nodes}, \ie vertices for which
we do not assign any labels and which do not contribute to the total size of
the graph.
We define the family \( \Link[\fA]\) as the family
obtained from \( \fA \) by replacing the two empty nodes by
complementary literals.
\end{definition}

\begin{lemma}
    \label{lemma:complementary:literal:insertion:operator}
    If \( \eA(z) \) is the EGF of a given family \( \fA \) of digraphs with
    even number of nodes, two distinguished empty nodes (of size 0),
    then the EGF of \(
    \Link[\fA]\) is given by
    \begin{equation}
        \eLink [\eA](z) := z\int_0^z \eA(t) dt
        .
    \end{equation}
\end{lemma}
\begin{proof}
    The number of graphs of size \( (2n-2) \) from \( \fA \) is
    \( (2n-2)! [z^{2n-2}] \eA(z) \). Then, the number of ways to insert the
    complementary labels into two distinguished positions is \( 2n \). The total
    number of graphs of size \( 2n \) from the family
    \( \Link[\fA] \) is therefore \( 2n \cdot (2n - 2)! [z^{2n}] z^2 A(z) =
    \frac{1}{2n - 1} (2n)! [z^{2n-2}] \eA(z) \). We conclude by noting that \(
    (2n) \)-th coefficient of \( z \int_0^z \eA(t) dt \) equals to \( \frac{1}{2n -
    1} [z^{2n - 2}] \eA(z) \) and multiplying by \( (2n)! \) gives the total
    number of graphs of size \( 2n \) from \( \Link[\fA] \).
\end{proof}

\begin{remark}
    It is possible to insert more than one pair of complementary literals, in
    this case more than one application of \( \Link \) is required. Naturally,
    in this case, it is required that the number of unlabelled empty slots
    should be equal to two times
    the number of inserted complementary literals.
\end{remark}

\subsection{Saddle point lemma for random graphs}
\label{subsection:sp:lemma:random:graphs}

Among many tools introduced in~\cite{janson1993birth}, the central one was the
saddle point lemma allowing to study the fine structure of a random graph near
the point of its phase transition using the generating functions.  We give it in
a slightly reformulated form without a proof.
\begin{lemma}[{\cite[Lemma 3]{janson1993birth}}]
    \label{lemma:sp:integration}

    Suppose that \( m = \frac{n}{2}(1 + \mu n^{-1/3}) \), and
    \( H(t) \) is a function analytic at \( t = 1 \).
    Then,
    as \( \mu \to -\infty \), and \( |\mu| \leq n^{1/12} \),
    \begin{equation}
        \dfrac{n!}{|\mathcal G(n,m)|} [z^n] \dfrac{U(z)^{n-m}}{(n-m)!}
        e^{V(z)}
        \dfrac{H(T(z))}{(1 - T(z))^y}
        =
        H(1)
        \left(\frac{n^{1/3}}{|\mu|}\right)^y
        (1 + O(|\mu|^{-3}))
    .
    \end{equation}
\end{lemma}

In order to apply the above lemma, it is needed to construct a graph from trees,
unicycles and a finite number of complex components as from ``builing bricks'',
and then extract the asymptotic proportion of such graphs.

\begin{remark}
For the subcritical phase \( \mu \to -\infty \), the condition \( |\mu| \leq
n^{1/12} \) can be replaced by \( |\mu| \ll n^{1/3} \), as shown
in~\cite{herve2011random} by a refined
technical analysis of the asymptotics. We do not apply this refinement in the
current paper as it would require a certain amount of additional technical details.

In a similar context, the enumeration of connected graphs has been pushed even
further in~\cite{flajolet2004airy} using purely analytic methods. The resulting
expansions are very tightly related with the Airy function and complex contour
integration in the context of the previous result~\cite{janson1993birth}. The
properties of the Airy function and the phenomenon of coalescing saddles
in combinatorics
are out of the scope of the current paper, and the interested reader can find
additional details \eg in~\cite{aldous1997brownian,banderier2001random}.
\end{remark}

\begin{remark}
    \label{remark:bicyclic:components}
Let us present an exemplary application of~\cref{lemma:sp:integration}
analogous to presented in~\cite{janson1993birth}.
If a simple graph with \( n \) vertices and \( m \) edges
contains only trees and unicycles, then the number of trees is equal to \( m-n
\), because each tree has one edge less than its number of vertices. The EGF for
such graphs is then \( e^{V(z)} U(z)^{n-m} / (n-m)! \).
The factor \( e^{V(z)} \) can be then rewritten as
\[
    e^{V(z)} = \dfrac{1}{\sqrt{1 - T(z)}} e^{-T(z)/2 - T^2(z)/4}
    .
\]
When \( m = \frac{n}{2}(1 + \mu n^{-1/3}) \),
\( \mu \to -\infty \) with \( n \), and \( |\mu| \leq n^{1/12} \),
the probability that a graph consists only of trees and unicycles, is, according
to~\cref{lemma:original:contour:integration}:
\[
    \dfrac{n!}{|\mathcal G(n,m)|} [z^{n}]
    \dfrac{U(z)^{n-m}}{(n-m)!}
    e^{V(z)}
    =
    1 - \dfrac{5 + o(1)}{24 |\mu|^{3}}
    .
\]
The interpretation of the factor \( 5/24 \) is at least twofold.
Formally, it appears as the second expansion term in the saddle point lemma with
\( 5/24 = \left.\frac{3y^2 + 3y - 1}{6}\right|_{y = 1/2} \).
At the same time, it denotes the sum of the compensation factors of the two
possible cubic bicyclic multigraphs (see~\cref{figure:bicyclic:kernels}), also
expressed as the sum of the inverse cardinalities of their automorphism groups.
Furthermore, the probability of having only one complex component which is
bicyclic, is \( \frac{5}{24} |\mu|^{-3} + O(|\mu|^{-6}) \), and the probability
of having a complex component of excess \( r \) is \( e_{r} |\mu|^{-3r} \),
where \( e_r = \frac{(6r)!}{2^{5r} 3^{2r} (3r)! (2r)!} \) corresponds to the sum
of compensation factors of cubic multigraphs of excess \( r \), see
\eg~\cite[Chapter 2]{bollobasrandomgraphs}. For complete asymptotic expansions
in powers of \( |\mu|^{-3} \) we refer to~\cref{section:full:asymptotic:expansion}.
\end{remark}

We give a variant of~\cref{lemma:sp:integration} for the case of directed graphs.

\begin{lemma}
    \label{lemma:sp:digraphs}

    As \( \mu \to -\infty \), and \( |\mu| \leq n^{1/12} \),
    \begin{equation}
        \dfrac{(2n)!}{|\mathcal D(2n,m)|} [z^{2n}]
        \dfrac{U_\to(z)^{2n-m}}{(2n-m)!}
        e^{V_\to(z)}
        \dfrac{H(T_\to(z))}{(1 - 2T_\to(z))^y}
        =
        H(1/2)
        \left(\frac{n^{1/3}}{|\mu|}\right)^y
        (1 + O(|\mu|^{-3}))
    \end{equation}
    \( H(t) \) is a function analytic at \( t = 1/2 \).
\end{lemma}

\begin{proof}
    As \( \frac{m}{n} \to 1 \) with \( n \) going to infinity, the number of
    simple digraphs with \( 2n \) vertices and \( m \) edges can be
    asymptotically related to the number of simple graphs with the same number
    of vertices and edges using the Stirling's formula
    \[
        \dfrac{|\mathcal D(2n,m)|}
        {2^m |\mathcal G(2n,m)|}
        = \dfrac{
           { 2 {n \choose 2} \choose m}
        }{
            2^m { {2n \choose 2} \choose m}
        }
        \to e^{1/8}
        , \quad
        \text{as } n \to \infty.
    \]
    Replacing \( |\mathcal D(2n,m)| \) by its asymptotic equivalent,
    \( U_\to(z) \) and \( T_\to(z) \) by
    \( \frac12 U(2z) \) and \( \frac12 T(2z) \), and also
    \( V_\to(z) \) by \( V(2z) + T(2z)^2 / 8 \), we obtain the expression
    \[
        e^{-1/8}\dfrac{(2n)! 2^{-m}}{|\mathcal G(2n,m)|} [z^{2n}]
        \dfrac{2^{-2n+m} U(2z)^{2n-m}}{(2n-m)!}
        e^{V(2z) + T(2z) / 8}
        \dfrac{H(T(2z)/2)}{(1 - T(2z))^y}
        .
    \]
    After substituting \( 2z = x \), we get an additional \( 2^{2n} \) from the
    operator of coefficient extraction \( [z^{2n}] \), and therefore, all the
    powers of two cancel out. We proceed
    by applying~\cref{lemma:sp:digraphs} and replacing each \( T(x) \) by \( 1
    \). Since the lemma is designed for a different regime, namely
    \( m = \frac{n}{2}(1 + \mu n^{-1/3}) \),
    we can adapt by replacing \( n
    \mapsto 2n \), \( \mu \mapsto 2^{1/3} \mu \).
    Finally, the factor \( e^{-1/8} \) also cancels out, because of the presence
    of the multiple \( e^{T(x) / 8} \), and the ratio of \( (2n)^{1/3} \) to
    \( 2^{1/3} |\mu| \) becomes again \( n^{1/3} / |\mu| \).
\end{proof}

\section{Extracting the asymptotics}
\label{section:asymptotic:analysis}

\subsection{Contradictory components in simple digraphs}
Before constructing the sum-representation digraphs with marked directed
subgraphs, we start by constructing the subcritical \emph{simple} digraphs first.
We are going to mark the same contradictory patterns, but without an additional
assumption that the paths are strictly distinct and that the digraph is conflict
free, and without excluding the edges \( x \to \n x \).

As we will see later in~\cref{subsection:excluding:non:sum:representations},
this set has a very similar structure to the set of sum-representation digraphs
\( \mathcal D^\circ(2n, m) \): the number of edge conflicts in a random digraph
\( D \in \mathcal D(2n, m) \) asymptotically follows a Poisson distribution with
parameter \( 1/8 \). Excluding edge conflicts is equivalent to conditioning on the
value of this Poisson variable being zero.

We start by demonstrating a result directly related to the probability of
satisfiability in the subcritical phase, which is shown by Kim
in~\cite{kim2008finding} to be asymptotically
\[
    \mathbb P(F \in \mathcal F(n,m)\text{ is SAT})
    =
    1 - \frac{1 + o(1)}{16 |\mu|^3}
\]
for \( m = n (1 + \mu n^{-1/3}) \), as \( \mu \to -\infty\) with \( n \),
and \( |\mu| = o(n^{1/3}) \). In this section we give a new explanation of the
factor \( 1/16 \) and show how to extend this result
to obtain a complete asymptotic distribution in powers of \( |\mu|^{-3} \).

The contradictory pattern that we are going to identify first, takes a form of
three sequences \( x \rightsquigarrow \n x \rightsquigarrow y \rightsquigarrow
\n y \) (\cref{fig:second:contradictory:pattern}). We start by considering the
case when this pattern is a part of a weakly connected component which is a
tree (\cref{fig:anti:pruning}).

\begin{lemma}
\label{lemma:contradictory:simple:digraphs}
Consider simple digraphs \( G \in \mathcal D(2n, m) \) with a
distinguished contradictory pattern \( x \rightsquigarrow \n x \rightsquigarrow
y \rightsquigarrow \n y \) such
that the weakly connected component containing this pattern is a tree,
and there are no weakly connected components whose non-oriented projections have
positive excess. The proportion of such digraphs among \( \mathcal D(2n,m) \)
where each such digraph is taken with weight \( 2^\ell \), where \( \ell \) is
the length of the pattern, is
\begin{equation}
    \dfrac
        {(2n)!}
        {|\mathcal D(2n, m)|}
    [z^{2n}]
    z \! \int \! z \! \int
    \left[
        \dfrac
            {U_\to(z)^{2n-m-1}}
            {(2n-m-1)!}
        e^{V_\to(z)}
        \dfrac
            {8 T_\to(z)^4 z^{-4}}
            {(1 - 2 T_\to(z))^3}
    \right]
    dz^2
    =
    \dfrac{1}{2 |\mu|^3} + O(|\mu|^{-6})
    .
\end{equation}
\end{lemma}

\begin{proof}
The weakly connected component containing a distinguished pattern \( x
\rightsquigarrow \n x \rightsquigarrow y \rightsquigarrow \n y \) can be
represented as three sequences of trees, each sequence of length at least one,
and an additional tree (\cref{fig:anti:pruning}). The generating function for
one such sequence, equipped with a weight \( 2^\ell \), where \( \ell \) is the
length of this sequence, is
\[
    \sum_{\ell \geq 1} T_\to(z)^\ell 2^\ell = \dfrac{2 T_\to(z)}
    {1 - 2 T_\to(z)}
    .
\]
Taking three such sequences and adding a triple multiple \( 2 \)
corresponding to linking three directed edges between the sequences, and by
adding the last tree, we get an EGF
\( \frac{8 T_\to(z)^4}{(1 - 2 T_\to(z))^3} \).
The next step is to divide by \( z^4 \) which corresponds to erasing the labels
of the four distinguished nodes (which will later become the labels \( x \),
\( \n x \), \( y \), and \( \n y \)).

\begin{figure}[hbt]
\centering
\begin{tikzpicture}[>=stealth',thick, node distance=1.0cm]
\draw
node[arnBleuPetit, minimum height = 1.2em, very thick](x)
                    at (0,0)           {\( x \)}
node[arnBleuPetit](x1) at (0,-1) {}
node[arnBleuPetit](x11) at (0,-2) {}
node[arnBleuPetit](a) at (1,0) {}
node[arnBleuPetit](a1) at (1,-1) {}
node[arnBleuPetit](a11) at (0.5,-2) {}
node[arnBleuPetit](a12) at (1,-2) {}
node[arnBleuPetit](a13) at (1.5,-2) {}
node[arnBleuPetit](s) at (2,0) {}
node[arnBleuPetit](s1) at (1.7,-1) {}
node[arnBleuPetit](s2) at (2.3,-1) {}
node[arnRougePetit, minimum height = 1.2em, very thick](x')
                    at (3,0)           {\( \fn x \)}
node[arnBleuPetit](d) at (4,0) {}
node[arnBleuPetit](d1) at (4,-1) {}
node[arnBleuPetit](d11) at (3.5,-2) {}
node[arnBleuPetit](d12) at (4.5,-2) {}
node[arnBleuPetit](f) at (5,0) {}
node[arnBleuPetit](f1) at (5,-1) {}
node[arnBleuPetit, minimum height = 1.2em, very thick](y)
                    at (6,0)           {\( y \)}
node[arnBleuPetit](y1) at (6,-1) {}
node[arnBleuPetit](y11) at (5.5,-2) {}
node[arnBleuPetit](y12) at (6.5,-2) {}
node[arnBleuPetit](y2) at (6.5,-1) {}
node[arnBleuPetit](g) at (7,0) {}
node[arnBleuPetit](g1) at (7,-1) {}
node[arnRougePetit, minimum height = 1.2em, very thick](y')
                    at (8,0)           {\( \fn y \)}
node[arnBleuPetit](y'1) at (8,-1) {}
;
\fromto{x}{a}     [very thick][blackred, thick][4][1];
\fromto{a}{s}     [very thick][blackred, thick][-4][-1];
\fromto{s}{x'}    [very thick][blackred, thick][4][1];
\fromto{x'}{d}    [very thick][blackred, thick][4][1];
\fromto{d}{f}     [very thick][blackred, thick][-4][-1];
\fromto{f}{y}     [very thick][blackred, thick][4][1];
\fromto{y}{g}     [very thick][blackred, thick][-4][-1];
\fromto{g}{y'}    [very thick][blackred, thick][4][1];
\sedge{x}{x1}     [][4];
\sedge{x1}{x11}   [][4];
\sedge{a}{a1}     [][4];
\sedge{a1}{a11}   [][4];
\sedge{a1}{a12}   [][4];
\sedge{a1}{a13}   [][4];
\sedge{s}{s1}     [][4];
\sedge{s}{s2}     [][4];
\sedge{d}{d1}     [][4];
\sedge{d1}{d11}   [][4];
\sedge{d1}{d12}   [][4];
\sedge{f}{f1}     [][4];
\sedge{y}{y1}     [][4];
\sedge{y}{y2}     [][4];
\sedge{y1}{y11}   [][4];
\sedge{y1}{y12}   [][4];
\sedge{g}{g1}     [][4];
\sedge{y'}{y'1}   [][4];
\end{tikzpicture}
\caption{\label{fig:anti:pruning}The case when the weakly connected component of
the contradictory pattern is a tree.}
\end{figure}
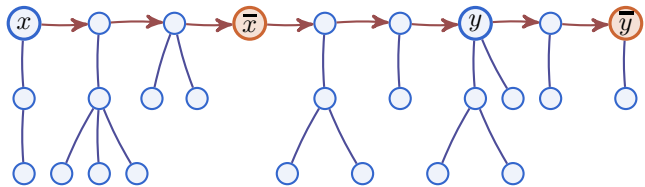

The digraphs required in the current lemma are obtained as product of the family
of directed trees and unicycles and the constructed contradictory pattern,
through literal linking construction
(see~\cref{definition:literal:linking:construction}).
By applying the complementary label
insertion operation \( \Lambda [A](z)  = z \! \int \! A(z) dz \)
from~\cref{lemma:complementary:literal:insertion:operator} twice, we obtain the
desired EGF.  The (weighted) number of digraphs is obtained by taking the \(
(2n) \)th coefficient of the EGF.

Since the number of trees in the forest of directed trees (without the
distinguished one) is \( (2n-m-1) \), the EGF of digraphs containing this forest
and set of directed unicycles is then
\[
        \dfrac
            {U_\to(z)^{2n-m-1}}
            {(2n-m-1)!}
        e^{V_\to(z)}
        ,
\]
and multiplication by the EGF of the distinguished tree with removed nodes
\(
\frac
    {8 T_\to(z)^4 z^{-4}}
    {(1 - 2 T_\to(z))^3}
\) and double application of the complementary label insertion operator finishes
the construction.

The asymptotic analysis of the expression is done in two steps.
Firstly, we get rid of the operator \( \Lambda \) by using the properties
\[
    [z^n] z F(z) = [z^{n-1}] F(z),
    \quad
    [z^n] \int F(z) dz = \dfrac{1}{n}[z^{n-1}] F(z).
\]
We then obtain
\begin{align*}
    &
    [z^{2n}]
    z \! \int \! z \! \int
    \left[
        \dfrac
            {U_\to(z)^{2n-m-1}}
            {(2n-m-1)!}
        e^{V_\to(z)}
        \dfrac
            {8 T_\to(z)^4 z^{-4}}
            {(1 - 2 T_\to(z))^3}
    \right]
    dz^2
    \\ &
    =
    \dfrac{1}{(2n-1)(2n-3)}
    [z^{2n-4}]
        \dfrac
            {U_\to(z)^{2n-m-1}}
            {(2n-m-1)!}
        e^{V_\to(z)}
        \dfrac
            {8 T_\to(z)^4 z^{-4}}
            {(1 - 2 T_\to(z))^3}
    \\ &
    =
    \dfrac{1}{(2n-1)(2n-3)}
    [z^{2n}]
        \dfrac
            {U_\to(z)^{2n-m-1}}
            {(2n-m-1)!}
        e^{V_\to(z)}
        \dfrac
            {8 T_\to(z)^4}
            {(1 - 2 T_\to(z))^3}
            .
\end{align*}
The second step is to apply~\cref{lemma:sp:digraphs} taking \( H(t) \) such that
\( H(T_\to(z)) = 8 T_\to(z)^4 (2n-m) U_\to(z)^{-1} \), \( y = 3 \).
Since \( U_\to(z) = T_\to(z) - T_\to(z)^2 \), the resulting value \( H(1/2) \)
will be \( 8/16 \cdot (2n-m) \cdot 4 \sim 2n \) when \( m/n \to 1 \).
An application of the lemma gives
\begin{align*}
&
    \dfrac{(2n)!}{|\mathcal D(2n,m)|}
    \dfrac{1}{(2n-1)(2n-3)}
    [z^{2n}]
    \dfrac{U_\to(z)^{2n-m}}{(2n-m)!}
    e^{V_\to(z)}
    \dfrac{H(T_\to(z))}{(1 - 2T_\to(z))^3}
    \\
    &
    =
    \dfrac{2n (1 + \mu n^{-1/3})}
    {(2n-1)(2n-3)}
    \left(\frac{n^{1/3}}{|\mu|}\right)^3
    (1 + O(|\mu|^{-3}))
    = \dfrac{1}{2 |\mu|^3} + O(|\mu|^{-6})
    .
\end{align*}
\end{proof}

The weakly connected component containing the pattern \( x \rightsquigarrow \n x
\rightsquigarrow y \rightsquigarrow \n y \) may happen to be a unicycle
(see~\cref{fig:weak:component:unicycle}), or a complex component of excess at
least \( 1 \)
(see~\cref{fig:weak:component:bicycle}). The constructions analogous
to~\cref{lemma:contradictory:simple:digraphs} counting simple digraphs with a
distinguished contradictory pattern weighted as \( 2^\ell \), where \( \ell \)
is the length of the pattern, yield asymptotics of order \( O(|\mu|^{-6}) \) or
smaller. This concept can be illustrated by the unicyclic case.

Instead of the three sequences with the generating function \( \frac{1}{(1 - 2
T_\to(z))^3} \) we consider 6 sequences (possibly empty). There are several
different mutual configurations of the distinguished contradidctory pattern and
the unicycle which contains this pattern, but all can be described by \( 6 \)
sequences. For one of these sequences the directions of the edges are not fixed,
but its length is not accounted in the weight of the graph. Therefore, the
generating function for such a sequence is also proportional to \( \frac{1}{1 -
2 T_\to(z)} \), where a factor \( 2 \) in front of \( T_\to(z) \) stands for the
choice of orientation of each edge. The number of trees in such a graph
remains \( (2n-m) \), and so, the contribution obtained
from~\cref{lemma:sp:digraphs} has order
\( \Theta\left(\frac{1}{(2n-1)(2n-3)} \cdot \frac{n^2}{|\mu|^{6}}\right)
    = \Theta(|\mu|^{-6}) \).

\begin{figure}[hbt]
    \begin{minipage}[t]{0.46\textwidth}
\centering
\begin{tikzpicture}[>=stealth',thick, node distance=1.0cm]
\draw
node[arnBleuPetit, minimum height = 1.2em, very thick](q1)
    at (360/8*1 : 1.2)  {\( y \)}
node[arnRougePetit, minimum height = 1.2em, very thick](q2)
    at (360/8*3 : 1.2)  {\( \fn x \)}
node[arnVertPetit](q3) at (360/8*5 : 1.2)  {}
node[arnBleuPetit, minimum height = 1.2em, very thick](q31)
    at ($(q3)+(360/8*3 : 1.2)$)  {\( x \)}
node[arnVertPetit](q4) at (360/8*7 : 1.2)  {}
node[arnRougePetit, minimum height = 1.2em, very thick](q41)
    at ($(q4)+(360/8*7 : 1.2)$)  {\( \fn y \)}
;
\fromto{q31}{q3}   [blackred, ultra thick][very thick][10][3][0.6];
\fromto{q4}{q41}   [blackred, ultra thick][very thick][-10][-3][0.6];
\fromto{q3}{q2}    [blackred, ultra thick][very thick][-30][-5][0.6];
\fromto{q2}{q1}    [blackred, ultra thick][very thick][-30][-5][0.6];
\fromto{q1}{q4}    [blackred, ultra thick][very thick][-30][-5][0.6];
\sedge{q4}{q3}     [][-30];
\end{tikzpicture}
\caption{\label{fig:weak:component:unicycle}Contradictory path in component of excess \( 0 \).}
    \end{minipage}\hfill
    \begin{minipage}[t]{0.46\textwidth}
        \centering
\begin{tikzpicture}[>=stealth',thick, node distance=1.0cm]
\draw
node[arnBleuPetit, minimum height = 1.2em, very thick](q1)
    at (30*1 : 1.2)  {\( y \)}
node[arnVertPetit](q2) at (30*3 : 1.2)  {}
node[arnRougePetit, minimum height = 1.2em, very thick](q3)
    at (30*5 : 1.2)  {\( \fn x \)}
node[arnVertPetit](q4) at (30*7 : 1.2)  {}
node[arnVertPetit](q5) at (30*9 : 1.2)  {}
node[arnVertPetit](q6) at (30*11 : 1.2)  {}
node[arnBleuPetit, minimum height = 1.2em, very thick](q41)
    at ($(q4)+(30*5 : 1.2)$)  {\( x \)}
node[arnRougePetit, minimum height = 1.2em, very thick](q61)
    at ($(q6)+(30*11 : 1.2)$)  {\( \fn y \)}
;
\fromto{q41}{q4}   [blackred, ultra thick][very thick][10][3][0.6];
\fromto{q6}{q61}   [blackred, ultra thick][very thick][-10][-3][0.6];
\fromto{q4}{q3}    [blackred, ultra thick][very thick][-10][-5][0.6];
\fromto{q3}{q2}    [blackred, ultra thick][very thick][-10][-5][0.6];
\fromto{q2}{q1}    [blackred, ultra thick][very thick][-10][-5][0.6];
\fromto{q1}{q6}    [blackred, ultra thick][very thick][-10][-5][0.6];
\sedge{q6}{q5}     [][-10];
\sedge{q5}{q4}     [][-10];
\sedge{q2}{q5}     [][-3];
\end{tikzpicture}
\caption{\label{fig:weak:component:bicycle}Contradictory path in component of excess \( 1 \).}
    \end{minipage}\hfill
\end{figure}

\begin{remark}
The same type of reasoning is used to show that the terms \( \mathbb E \xi_r \)
with \( r \geq 2 \) and \( \mathbb E \xi (\xi - 1) \cdots (\xi - k + 1) \) for
\( k \geq 2 \) do not contribute to the term of order \( \Theta(|\mu|^{-3}) \)
and start contributing from \( \Theta(|\mu|^{-6}) \). Refinement of this
technique using the complete asymptotic expansion of the saddle point lemma
(see the refinement of~\cref{lemma:sp:integration}
in~\cref{section:full:asymptotic:expansion})
can be used to obtain the complete asymptotic
expansions of the probability of satisfiability in powers of \( |\mu|^{-3} \).
The problem of enumeration of the mutual combinations of tuples of minimal
contradictory components inside an implication digraph, in order to compute the
factorial moments of \( \xi \), seems to be a challenging task.
\end{remark}

\subsection{From simple digraphs to sum-representations}
\label{subsection:excluding:non:sum:representations}

The two details peculiar to sum-representations that were not treated in the
previous section are the following: firstly, all the literals of the paths \( x
\rightsquigarrow \n x \), \( \n x \rightsquigarrow y \), \( y \rightsquigarrow
\n y \) of the distinguished pattern, except \( x \) and \( \n x \),
and \( y \) and \( \n y \), should be pairwise strictly distinct; secondly, there
should be no edge conflicts, \ie pairs of complementary edges. Note that a
presence of an edge \( x \to \n x \) automatically induces a conflict, so there
should be no such edges as well.

Excluding these cases requires inclusion-exclusion: in order to count the
instances not containing conflicts and complemetary literal pairs, we count the
instances with distinguished conflicts and distinguished complementary literal
pairs, and then take the alternating sum over them. The two
inclusions-exclusions can be done independently.

\begin{lemma}[Excluding complementary literals on the paths]
    \label{lemma:sd}
    Among the digraphs with a distinguished pattern \( x \rightsquigarrow \n x
    \rightsquigarrow y \rightsquigarrow \n y \), taken with weight \( 2 \) to
    the power of the length of the pattern, the asymptotic proportion of
    digraphs in which there exist two literals on the pattern which are not
    strictly distinct (except the pairs \( x \) and \( \n x \), and \( y \) and
    \( \n y \)), is only \( O(n^{-1/3} |\mu|^{-2}) \).
\end{lemma}

\begin{proof}
    Let a random variable \( X \) denote the number of pairs of complementary
    literals on the distinguished pattern, not counting \( x \) and \( \n x \),
    and \( y \) and \( \n y \). Using the inclusion-exclusion principle, we can
    express the probability of the event \( [X = 0] \) as
    \[
        \mathbb P(X = 0) = 1 - \mathbb E X + \dfrac{1}{2!} \mathbb E X(X-1) -
        \cdots,
    \]
    where \( \mathbb E X \) corresponds to the proportion of digraphs having a
    distinguished pattern and a distinguished additional pair of complementary
    literals \( z \) and \( \n z \). Further terms correspond to distinguishing
    several pairs of complementary literals, etc.

    By marking two complementary literals in the case when a contradictory
    pattern is a tree (see~\cref{lemma:contradictory:simple:digraphs}
    and~\cref{fig:anti:pruning}), we obtain \( 5 \) sequences instead of \( 3
    \), and we need three applications of the label insertion operator \(
    \Lambda \) instead of just two. We then obtain a generating function
    \begin{equation}
        F(z)
        =
        \Lambda^3 \left(
            \dfrac{
                U_{\to}(z)^{2n-m-1}
            }{
                (2n-m-1)!
            }
            e^{V_\to(z)}
            \dfrac{2^5 T_\to(z)^6 z^{-6}}{(1 - 2 T_\to(z))^5}
        \right)
    \end{equation}
    and after extracting the asymptotics using~\cref{lemma:sp:digraphs}, we
    obtain
    \[
        \mathbb E X
        =
        \dfrac{
        [z^{2n}]
        \Lambda^3 \left(
            \dfrac{
                U_{\to}(z)^{2n-m-1}
            }{
                (2n-m-1)!
            }
            e^{V_\to(z)}
            \dfrac{2^5 T_\to(z)^6 z^{-6}}{(1 - 2 T_\to(z))^5}
        \right)
        }
        {[z^{2n}]
        \Lambda^2 \left(
            \dfrac{
                U_{\to}(z)^{2n-m-1}
            }{
                (2n-m-1)!
            }
            e^{V_\to(z)}
            \dfrac{2^3 T_\to(z)^4 z^{-4}}{(1 - 2 T_\to(z))^3}
        \right)
        }
        =
        \Theta(n^{-1/3} |\mu|^{-2})
    \]
    The same factor \( n^{-1/3} \) multiplied by a polynomial of \( |\mu|^{-1}
    \) and \( n^{-1/3} \) appears when the weakly connected component containing
    the contradictory pattern is not necessarily a tree.
    Further factorial moments of \( X \) will have respective contributions of
    order \( n^{-k/3} \), \( k \geq 2 \).
    From~\cite{janson1993birth} it is known that the components with infinite
    excess appear with an asymptotic zero probability, therefore, the interiors
    of the distinguished paths have pairwise strictly distinct literals with a
    probability of \( 1 - O(n^{-1/3} |\mu|^{-2}) \).
\end{proof}

As opposed to the first inclusion-exclusion on the number of complementary
literals in the distinguished pattern,
it can be shown that the number of complementary edge pairs in the whole graph
follows a limiting Poisson distribution with parameter \( 1/8 \), so the absence
of conflicting edges comes at a positive probability \( e^{-1/8} \). In this
paper, we focus only on the probability of not having edge conflicts.

\begin{lemma}[Excluding edge conflicts]
    With an asymptotic probability \( e^{-1/8} \), the digraphs
    from \( \mathcal D(2n, m) \)
    with a
    distinguished pattern \( x \rightsquigarrow \n x \rightsquigarrow y
    \rightsquigarrow \n y \) weighted according to \( 2 \) to the power of its
    length, are conflict-free.
\end{lemma}

\begin{proof}
    Following the same inclusion-exclusion principle, we introduce a random
    variable \( X \) denoting the number of conflicting pairs of edges. The
    probability of the event \( [X = 0] \) can be expressed using the
    inclusion-exclusion principle.

    Let us start by computing \( \mathbb E X \), \ie the expected number of the
    conflicting edges. According to the linearity of the mathematical
    expectation, \( \mathbb E X \) equals to the total number of possible
    complementary edge pairs times the probability that one distinguished
    edge pair, say \( 1 \to 2 \), \( \n 2 \to \n 1 \) is involved in a conflict.

    Using~\cref{lemma:complementary:literal:insertion:operator},
    we construct the EGF for the family of digraphs containing
    a distinguished conflict \( x \to y \), \( \n y \to \n x \) by inserting
    two pairs of labels \( x, \n x \) and \( y, \n y \) at the free slots.
    This insertion can happen in several different ways: each of the edges can
    belong to the forest of directed trees, to the set of unicycles, and to the
    complex component (or to the component containing the distinguished pattern
    \( x \rightsquigarrow \n x \rightsquigarrow y \rightsquigarrow \n y\)).

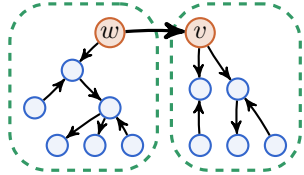
\begin{figure}[hbt]
\centering
\begin{tikzpicture}[>=stealth',thick, node distance=1.0cm]
\draw
node[arnRougePetit, minimum height = 1.1em, thick](R)
    at (0, 0) {\( v \)}
node[arnRougePetit, minimum height = 1.1em, thick](a)
    at ($(R)-(1.2,0)$) {\( w \)}
node[arnBleuPetit](a1)     at ($(a)   + (-.5,-0.5)$)            {}
node[arnBleuPetit](a11)    at ($(a1)  + (0.5,-0.5)$)           {}
node[arnBleuPetit](a111)   at ($(a11) - (0.2,0)+(0.5, -0.5)$) {}
node[arnBleuPetit](a112)   at ($(a11) - (0.2,0)+(0,   -0.5)$) {}
node[arnBleuPetit](a113)   at ($(a11) - (0.2,0)+(-0.5,-0.5)$) {}
node[arnBleuPetit](a12)    at ($(a1)  + (-.5,-0.5)$)           {}
node[arnBleuPetit](R1)     at ($(R)   + (0,-0.75)$)          {}
node[arnBleuPetit](R11)    at ($(R1)  + (0,-0.75)$)          {}
node[arnBleuPetit](R2)     at ($(R)   + (0.5,-0.75)$)        {}
node[arnBleuPetit](R21)    at ($(R2)  + (0,-0.75)$)          {}
node[arnBleuPetit](R22)    at ($(R2)  + (0.5,-0.75)$)          {}
;
\node[dashed,draw,fit=(R)(R211)(R212), very thick,
      rounded corners=5mm,inner sep= 5pt, bgreen] {};
\node[dashed,draw,fit=(a)(a12)(a11)(a111)(a113), very thick,
      rounded corners=5mm,inner sep= 5pt, bgreen] {};
\fromto{a}{R}     [ultra thick][very thick][-4][-1];
\fromto{a}{a1}     [thick][thick][4];
\fromto{a1}{a11}   [thick][thick][4];
\fromto{a111}{a11} [thick][thick][4];
\fromto{a11}{a112} [thick][thick][4];
\fromto{a11}{a113} [thick][thick][4];
\fromto{a12}{a1}   [thick][thick][4];
\fromto{R}{R1}     [thick][thick][4];
\fromto{R}{R2}     [thick][thick][4];
\fromto{R11}{R1}   [thick][thick][-4];
\fromto{R2}{R21}   [thick][thick][4];
\fromto{R22}{R2}   [thick][thick][4];
\end{tikzpicture}
\caption{\label{fig:conflict:edge}Digraph tree with a marked edge and two empty
slots.}
\end{figure}

    First, consider the case when each of the distinguished edges belongs to a
    distinct tree in a forest of directed trees. We denote the random variable
    corresponding to the number of such
    occurences as \( X_1 \) which we will show to be asymptotically equivalent
    to \( X \).
    A tree containing a
    distinguished edge \( v \to w \) can be represented as a pair of trees with
    removed roots (which will be later filled by certain labels),
    see~\cref{fig:conflict:edge}. The corresponding EGF for digraphs with two
    distinguished trees containing edges \( v \to w \), \( \n w \to \n v \), and
    containing a distinguished pattern \( x \rightsquigarrow \n x
    \rightsquigarrow y \rightsquigarrow \n y \) weighted as \( 2 \) to the power
    of the length of the pattern, is then equal to
    \begin{equation}
        \Lambda^4 \left[
            \dfrac{1}{2!}
            \left(
                \dfrac{T_\to(z)}{z}
            \right)^4
        \dfrac{U_\to(z)^{2n-m-3}}{(2n-m-3)!}
        e^{V_\to(z)}
        \dfrac{8 T_\to(z)^4 z^{-4}}{(1 - 2 T_\to(z))^3}
        \right]
        ,
    \end{equation}
    where two of the applications of the operator \( \Lambda \) stand for the
    insertion of the label pairs \( v \) and \( \n v \), and \( w \) and \( \n w
    \) for the distinguished conflict edge pair, and another two applications
    for the insertion of complementary literals \( x \) and \( \n x \), and
    \( y \) and \( \n y \) in a distinguished pattern. The term
    \( 1/2 \cdot T_\to(z)^4 z^{-4} \) is an EGF of the set of two trees with a
    distinguished edge. The cardinality \( (2n-m) \) of the forest of trees is
    then reduced by \( 3 \).

    The expected value of \( X_1 \) is then expressed as
    \[
        \mathbb E X_1 =
        \dfrac{
            [z^{2n}]
            \Lambda^4 \left(
                \dfrac{1}{2!}
                \left(
                    \dfrac{T_\to(z)}{z}
                \right)^4
            \dfrac{U_\to(z)^{2n-m-3}}{(2n-m-3)!}
            e^{V_\to(z)}
            \dfrac{8 T_\to(z)^4 z^{-4}}{(1 - 2 T_\to(z))^3}
            \right)
        }{
            [z^{2n}]
            \Lambda^2 \left(
            \dfrac{U_\to(z)^{2n-m-1}}{(2n-m-1)!}
            e^{V_\to(z)}
            \dfrac{8 T_\to(z)^4 z^{-4}}{(1 - 2 T_\to(z))^3}
            \right)
        }
        .
    \]
    This expression can be readily analysed
    by~\cref{lemma:complementary:literal:insertion:operator}.
    Using the property that the exchange of the operator \( [z^{2n}] \)
    with \( \Lambda \) is done according to the rule
    \( [z^{2n}] \Lambda F(z) = \frac{1}{2n-1}[z^{2n}] z^2 F(z) \),
    we conclude that the unique multiples in the numerator of \( \mathbb E X_1 \),
    if carefully counted, appear as follows:
    \begin{enumerate*}[label=(\textit{\roman*})]
        \item a multiple asymptotically equivalent to \( (2n)^{-2} \) from the
            double application of \( \Lambda \) and its exchange with \(
            [z^{2n}] \);
        \item a multiple \( 1/2 \) because the two trees with distinguished
            conflict form a set;
        \item a multiple \( 2^{-4} \) appear from \( T_{\to}(z)^4 \);
        \item a multiple \( 4^2 \) appears from \( U_\to(z)^2 \);
        \item a multiple \( n^2 \) from the change in the factorial \(
            (2n-m-1)! \) of the EGF of forest;
        \item all the occurences of \( z \) cancel out.
    \end{enumerate*}
    Collecting the powers of two, we finally obtain
    \[
        \mathbb E X_1 \sim
        2^{-2-1-4+4} =
        \dfrac{1}{8}.
    \]
    So far, we have considered only one particular case, but most of the work
    has been actually already done.

    Without caring too much for the exact coefficients in the asymptotics, let us
    count the expected values of the following
    random variables:
    \begin{enumerate*}[label=(\textit{\roman*})]
        \item \( X_2 \) for conflicting pairs of edges which are located in the
            same directed tree;
        \item \( X_3 \) for conflicting pairs, one located in a tree, another
            located in a unicycle (or a complex component);
        \item \( X_4 \) for conflicting pairs, neither is located in the forest.
    \end{enumerate*}
    For \( \mathbb E X_2 \), for the first conflicting edge we need to fix a
    pair of rooted trees, and assuming without loss of generality (but with a
    loss of the exact multiplicative constant) that the second conflicting edge
    is inside the first tree, we obtain a path from this edge to the root.
    Therefore, the EGF for such objects is proportional, up to a constant, to
    \(
        \frac{1}{1 - 2 T_\to(z)} (T_\to(z) / z)^4
    \), which corresponds to four trees without roots and a sequence of trees,
    each term of the sequence equipped with a weight \( 2 \) corresponding to
    the choice of the orientation of the edge in the sequence. Plugging it into
    the EGF of counted graphs yields
    \begin{equation}
        \mathbb E X_2 \sim
        \dfrac{C_2}{\Theta(|\mu|^{-3})}
            [z^{2n}]
            \Lambda^4 \left(
                \dfrac{1}{1 - 2 T_\to(z)}
                \left(
                    \dfrac{T_\to(z)}{z}
                \right)^4
            \dfrac{U_\to(z)^{2n-m-2}}{(2n-m-2)!}
            e^{V_\to(z)}
            \dfrac{8 T_\to(z)^4 z^{-4}}{(1 - 2 T_\to(z))^3}
            \right)
        .
    \end{equation}
    We immediately notice that using fewer than two trees costs an asymptotic
    multiple of \( n \) because of the factorial in the denominator for the EGF
    of forest \( U_\to(z)^{2n-m} / (2n - m)! \), while we gain only a multiple
    \( n^{1/3} \) for an additional sequence constructor. Thus, \( \mathbb E
    X_2 \) is negligible by a factor \( n^{2/3} \).
    The same story happens again for \( X_3 \) and \( X_4 \) since in all these
    cases fewer than two trees are used, and for each lost multiple \( n \) it
    is possible to obtain at most two sequence constructions, contributing each
    \( n^{1/3} \).

    Again, the same principle is applicable for computing the higher moments:
    the most probable situation is to have all the conflicting edges separately
    in distinct trees (otherwise the contribution will be negligible by a factor
    at least \( n^{2/3} \)), and in this case,
    \[
        \dfrac{1}{k!} \mathbb E X (X - 1) \cdots (X - k + 1)
        \sim \dfrac{1}{8^k k!}
        ,
    \]
    and therefore, as \( n \to \infty \), \( \mathbb P(X = 0) \to \sum_{k \geq 0}
    \frac{(-1)^k}{8^k k!} = e^{-1/8} \).
\end{proof}

\begin{corollary}[Formulation of the counting technique]
\label{corollary:counting:technique}
Note that, by simple asymptotic analysis,
\(
\frac
    {|\mathcal D(2n,m)|}
    {|\mathcal D^\circ(2n,m)|}
    \to e^{1/8}
\).
Hence, by following~\cref{lemma:sd}, we obtain that
the asymptotic probabilities obtained
by counting the (weighted) patterns in \emph{simple digraphs},
without excluding
conflicting edges and without taking care of strict distinctness of the
paths, are equal to the asymptotic probabilities that sum-representation
digraphs contain such (weighted) patterns.
\end{corollary}

\begin{remark}
As can be seen from the proof, the same technique can be applied to the
situations when instead of the pattern \( x \rightsquigarrow \n x
\rightsquigarrow y \rightsquigarrow \n y \) a different pattern is
distinguished. This can be either
used for the case when the expectations of \( \xi_r \) are counted, or
when
further factorial moments of \( \xi \) are being considered. In all such
cases, the edge conflicts will most probably appear in the forest component,
and excluding such conflicts will always give a multiple \( e^{-1/8} \).
\end{remark}

\begin{remark}
\label{remark:breaking:conditions}
Using the same techniques and the inclusion-exclusion method, it is possible
to consider the 2-SAT model where some of the
conditions~\ref{condition:no:double:edges}--\ref{condition:no:multigraphs}
are violated. Accordingly, in such models the loops or multiple edges may appear.
The number of the counted graphs will then be multiplied by a
certain constant appearing from inclusion-exclusion or from a different symbolic
construction,
while the total number of graphs in the denominator will
be coincidentally multiplied by the same constant. Therefore, the
probability of satistifability and its asymptotic expansion will remain
unchanged under the different models.
\end{remark}

\subsection{The structure of contradictory components}
Let us recall that the contradictory component in the implcation digraph,
defined as the set of contradictory variables \( x \rightsquigarrow \n x
\rightsquigarrow x \), forms a set of strongly connected components,
see~\cref{remark:set:cscc}. The following theorem gives a description
of the contradictory component in the subcritical phase.
\begin{theorem}
    \label{theorem:satisfiability}
    Suppose that \( m = n(1 + \mu n^{-1/3}) \), and \( \mu \to -\infty \) with
    \( n \) while remaining \( |\mu| \leq n^{1/12} \). The contradictory
    component of an implication digraph corresponding to a random formula \( F
    \in \mathcal F(n,m) \) has an excess \( r \) with probability
    \begin{equation}
        \mathbb P(\text{excess of the contradictory component} = r)
        = C_r |\mu|^{-3r} (1 + O(|\mu|^{-3})).
    \end{equation}
    Moreover, for every finite \( r \), the kernel of this contradictory component
    is cubic with probability \( 1 - O(n^{-1/3}) \).
    The coefficient \( C_r \) is equal to the sum of \( \sum_M 2^{-r}
    \varkappa(M) / (2r)! \) taken over all possible labelled cubic contradictory
    components of excess \( r \).
\end{theorem}

\begin{proof}
    Let \( \xi_r \) denote the random variable which equals to the number of
    contradictory components of excess \( r \) (each component is not
    necessarily connected and the different components may possibly overlap).
    Using a manipulation with formal power series that can be interpreted as a
    variation of the inclusion-exclusion principle, we can express the
    probability that \( \xi_r \) equals \( 1 \).

    Let \( F(x) := \sum_{k \geq 0} \mathbb P(\xi_r = k) x^k \) be the
    probability generating function (PGF) of \( \xi_r \). Then,
    \( F^{(k)}(z) = \mathbb E \xi_r (\xi_r-1)\cdots (\xi_r - k + 1) \).
    Using Taylor series expansion at \( x = 1 \), provided that \( F(x) \) is
    analytic in a circle of radius greater than \( 1 \), we obtain
    \begin{equation}
        F(x) = \sum_{k \geq 0}
        \dfrac{\mathbb E \xi_r(\xi_r-1) \cdots (\xi_r-k+1)}
        {k!} (x-1)^k
        .
    \end{equation}
    Using the fact that \( \mathbb P(\xi_r = 1) = F'(0) \), we obtain the
    expression
    \begin{equation}
        \mathbb P(\xi_r = 1)
        =
        \dfrac{
            \mathbb E \xi_r
        }{0!}
        -
        \dfrac{
            \mathbb E \xi_r (\xi_r - 1)
        }{1!}
        +
        \dfrac{
            \mathbb E \xi_r (\xi_r - 1) (\xi_r - 2)
        }{2!}
        - \cdots .
    \end{equation}
    Here, \( \mathbb E \xi_r \) can be rewritten as the number of implication
    digraphs with a distinguished contradictory component of excess \( r \),
    \( \mathbb E \xi_r(\xi_r - 1) \) is the number of implication digraphs with
    a distinguished pair of contradictory components, etc. The distinguished
    pair of contradictory components forms a contradictory component by itself,
    of excess at least \( (r + 1) \), as \eg in~\cref{example:overlap}.

    Let us choose a contradictory component \( \mathcal C \) of excess \( r \).
    We choose a \emph{contradictory pattern} \( \pi \) which is chosen by
    taking an arbitrary sum-representation of the kernel of \( \mathcal C \), so
    that no two complementary edges are chosen.
    Assume that all the isolated vertices are not included into the pattern, so
    that there might appear distinguished vertices that do not have their
    complementaries in \( \pi \).
    Finally, every edge of \( \pi \) is replaced by a sequence of directed trees
    to obtain a directed
    weakly connected component \( \mathcal P \).
    By combining the reasoning of~\cref{subsection:simplest:cscc} and using
    \cref{lemma:compensation:factor:interpretation}, we conclude that
    \( \mathbb E \xi_r \) is asymptotically worth the number of all
    sum-representation digraphs \( \mathcal D^\circ(2n, m) \) containing \(
    \mathcal P \), counted with weight \( 2^\ell \), where \( \ell
    \) is the number of edges of \( \pi \), and divided by the compensation
    factor of \( \mathcal C \).
    Note that such a weakly connected component \( \mathcal P \) may be a part
    of a larger weakly connected component in a sum-representation digraph, so
    it is required to take the sum over all possible weakly connected components
    conatining \( \mathcal P \).

    Let us define the following quantities:
    \begin{enumerate*}[label=(\textit{\roman*})]
        \item \( \tau(\pi) \) equal to the number of directed edges of \( \pi
            \);
        \item \( \nu(\pi) \) equal to the number of pairs of complementary
            variables in \( \pi \);
        \item \( \varphi(\pi) \) equal to the number of literals that do not
            have their complementaries in \( \pi \).
    \end{enumerate*}
    Then, the number of Boolean variables in \( \mathcal C \) equals \( \nu(\pi)
    + \varphi(\pi) \), and the excess \( r \) then equals
    \( r = \tau(\pi) - \varphi(\pi) - \nu(\pi) \). At the same time, the excess
    of the unoriented projection of \( \pi \) (in the sense of simple graphs) is
    equal to \( \tau(\pi) - \varphi(\pi) - 2 \nu(\pi) = r - \nu(\pi) \).

    The EGF \( f_{\mathcal P}(z) \) for digraphs \( D \in \mathcal D(2n,m) \)
    containing \( \mathcal P \) as a separate weakly connected components, can
    be expressed as
    \begin{equation}
        f_{\mathcal P}(z)
        =
        \dfrac{1}{\varkappa(\pi)}
        \Lambda^\nu
        \left(
            \dfrac{U_\to(z)^{2m-n + (r - \nu(\pi))}}
                {(2m-n + r - \nu(\pi))!}
            e^{V_\to(z)}
            \dfrac{ (T_\to(z))^{\varphi(\pi) + 2 \nu(\pi)}
            z^{-2 \nu(\pi)} 2^{\tau(\pi)}}
            {(1 - 2 T_\to(z))^\tau}
        \right) .
    \end{equation}
    In the case above, each of the literals of \( \pi \) is coloured into a
    separate distinguished colour, and only the compensation factor of \( \pi \)
    is considered.

    By analysing the asymptotics of \(
    \dfrac{(2n)!}{\mathcal D(2n,m)}
    [z^{2n}] f_{\mathcal P}(z)
    \)
    similarly to~\cref{lemma:contradictory:simple:digraphs}, we obtain
    \[
        \dfrac{(2n)!}{\mathcal D(2n,m)}
        [z^{2n}] f_{\mathcal P}(z)
        \sim
        \dfrac{1}{\varkappa(\pi)}
        \dfrac{1}{(2n)^{\nu(\pi)}}
        4^{\nu(\pi) - r}
        n^{\nu(\pi)-r}
        2^{-\varphi(\pi) - 2 \nu(\pi)}
        2^{\tau(\pi)}
        n^{\tau(\pi)/3}
        |\mu|^{-\tau(\pi)}
        .
    \]
    By using the relation \( r = \tau(\pi) - \varphi(\pi) - \nu(\pi) \), we
    obtain
    \[
        \dfrac{(2n)!}{\mathcal D(2n,m)}
        [z^{2n}] f_{\mathcal P}(z)
        \sim
        n^{\tau(\pi)/3 - r}
        |\mu|^{-\tau(\pi)}
        \dfrac{1}{\varkappa(\pi)}
        2^{-r}
        .
    \]
    If the kernel of the component \( \mathcal C \) is not a cubic multigraph,
    then \( r > \tau(\pi) / 3 \), and the contributions of such terms are
    negligible.  Otherwise, \( \tau(\pi) = 3 r \).
    In the case when \( \mathcal P \) is a part of a larger weakly connected
    component of higher excess, this results of asymptotic of order \(
    |\mu|^{-3r - 3} \) which is negligible compared to \( |\mu|^{-3r} \).

    Finally, let us obtain the asymptotics of \( \mathbb E \xi_r \).
    By~\cref{corollary:counting:technique}, enumeration inside simple digraphs
    gives the same asymptotics as the probability in sum-representations.

    The
    expected value \( \mathbb E \xi_r \) can be obtained by adding up the
    contributions of all possible contradictory components \( \mathcal C
    \) of excess \( r \) whose kernels are cubic.
    We denote such a contribution by \( \mathbb E
    \xi_{\mathcal C} \), where \( \xi_{\mathcal C} \) is the corresponding
    random variable. Then, recalling the reasoning
    from~\cref{subsection:simplest:cscc}, and by choosing a corresponding
    sum-representation \( \pi \) of the kernel of \( \mathcal C \),
    by using~\cref{lemma:compensation:factor:interpretation} and the fact
    that a cubic multigraph of excess \( r \) contains \( 2r \) vertices,
    we express \( \mathbb E \xi_{\mathcal C} \) as
\begin{align}
    \mathbb E \xi_{\mathcal C} &\sim
    \dfrac{\varkappa(\mathcal C)}{ (2r)! \varkappa(\pi) }
    \dfrac{
        \sum_{\ell = 0}^\infty
        2^\ell
        \left|\left\{
            (G, p_\ell) \ | \
            \substack{
                G \in \mathcal D^\circ(2n,m), \
                p_\ell \subset G \text{ is obtained from \( \pi \) by}
            \\
                \text{inserting sequences of
                trees of total length \( \ell \)
            }}
        \right\}\right|
    }{
        \big|\{
            G \mid
            G \in \mathcal D^\circ(2n,m)
        \}\big|
    }
    .
\end{align}
Taking the sum over all such \( \mathcal C \)
we obtain the dominant contribution of \( \mathbb E \xi_r \).
Further terms of the inclusion-exclusion for \( \mathbb P(\xi_r = 1) \) have
order at most \( |\mu|^{-3r-3} \) and are, therefore, negligible by a factor \(
|\mu|^3 \).
Collecting the dominant contributions, we conclude that
\begin{equation}
    \mathbb P(\xi_r = 1)
    \sim
    \ \
    \sum_{\mathclap{\substack{\mathcal C \text{ of excess }r\\
    \text{with cubic}\\ \text{kernels}}}}
    \ \
    \mathbb E \xi_{\mathcal C} \sim
    \sum_{\mathcal C}
    \dfrac{\varkappa(\mathcal C)}{2^r (2r)!}
    |\mu|^{-3r}
    .
\end{equation}
\end{proof}

We present a different proof of a theorem from~\cite{kim2008finding} using the
compensation factors of the contradictory components which comes as a corollary
of the above theorem.
\begin{corollary}
    For a random formula \( F \in \mathcal F(n,m) \), when \( m = n(1 + \mu
    n^{-1/3}) \) and \( \mu \to -\infty \) with \( n \)
    while remaining \( |\mu| \leq n^{1/12} \),
    \begin{equation}
        \mathbb P(F \text{ is satisfiable}) =
        \left(1 - \dfrac{1}{16 |\mu|^3}
        \right) (1 + O(\mu n^{-1/3} + \mu^{-3}))
        .
    \end{equation}
\end{corollary}
\begin{proof}
    {\setstretch{1.1}
    In the subcritical phase,
    the compensation factor of the only possible cubic contradictory implication
    multidigraph of excess \( 1 \)
    (viz.~\cref{fig:cscc:excess:one}) is equal to \( 1/4 \).
    Therefore,
    the probability of having a contradictory component of excess \( 1 \) is \(
    \frac{2^{-2}}{2!} \cdot \frac{1}{4} |\mu|^{-3} = \frac{1}{16} |\mu|^{-3}
    \). The probability of having a contradictory component of higher excess is
    then \( \Theta(|\mu|^{-6}) \), and so, is negligible.\par
}
\end{proof}

\subsection{Number of contradictory variables}

\begin{theorem}
    \label{theorem:contradictory:variables}
    Let \( m = n(1 + \mu n^{-1/3}) \), \( \mu \to -\infty \), \( |\mu| \leq
    n^{1/12} \).
    Assuming that the excess of the contradictory component is \( r \), and this
    component has a cubic kernel, the number of contradictory variables
    \( V_n \) in a random formula \( F \in \mathcal F(n,m) \) follows
    asymptotically a Gamma law with shape parameter \( 3r \) and scale
    parameter \( n^{1/3} |\mu|^{-1} \), so that
    \begin{equation}
        \lim_{n \to \infty}
        \mathbb P\Big(V_n = x n^{1/3} |\mu|^{-1}
        \ \Big| \ \text{excess } = r\Big)
        =
        \dfrac{x^{3r-1}}{\Gamma(3r)} e^{-x}
        .
    \end{equation}
\end{theorem}

\begin{proof}
    Fix a contradictory component \( \mathcal C \) of excess \( r \) with a
    cubic kernel.
    Construct a contradictory pattern \( \pi \) by taking an arbitrary
    sum-representation of the kernel of \( \mathcal C \), and replace every
    oriented edge of \( \pi \) by a sequence of directed trees, thus obtaining a
    digraph \( \mathcal P \). Consider an EGF
    \( F_{\mathcal P}(z, u) \)
    for directed graphs with a distinguished weakly connected component \(
    \mathcal P \) counted with weight \( 2^\ell \) where \( \ell \) denotes the
    length of the \( 2 \)-core of \( \mathcal P \). The variable \( u \) then
    marks all the contradictory variables on \( \mathcal C \). Then,
    using the similar constructions as in~\cref{theorem:satisfiability}, we
    express \( F_{\mathcal P}(z, u) \):
    \begin{equation}
        F_{\mathcal P}(z, u)
        =
        \dfrac{1}{\varkappa(\pi)}
        \Lambda^{2r}
        \left(
            \dfrac{U_\to(z)^{2m-n - r}}
                {(2m-n - r)!}
            e^{V_\to(z)}
            \dfrac{ (T_\to(z))^{4r}
            z^{-4r} 2^{3r}}
            {(1 - 2 u T_\to(z))^{3r}}
        \right) .
    \end{equation}
    The expected value of the number of contradictory variables conditioned on
    this pattern \( \mathcal P \) is then
    \begin{equation}
        \mathbb E [V_n \mid \mathcal P] =
        \dfrac{\partial_u[z^{2n}] F_{\mathcal P}(z, u)|_{u = 1}}
        {[z^{2n}] F_{\mathcal P}(z, 1)}
        \sim
        3r \cdot n^{1/3} |\mu|^{-1},
    \end{equation}
    and more generally, \( k \)-th factorial moment can be expressed as
    \begin{equation}
        \mathbb E [V_n \cdots (V_n-k+1) \mid \mathcal P] =
        \dfrac{\partial_u^k [z^{2n}] F_{\mathcal P}(z, u)|_{u = 1}}
        {[z^{2n}] F_{\mathcal P}(z, 1)}
        \sim
        \dfrac{\Gamma(3r+k)}{\Gamma(3r)!}
        (n^{1/3} |\mu|^{-1})^k
        .
    \end{equation}
    Note that the resulting factorial moments do not depend on the choice of \(
    \mathcal C \), therefore, these are also the asymptotic moments of the
    unconditioned variable \( V_n \).

The sequence of moments of the scaled random variable \( \widehat{V}_n := V_n
n^{-1/3} |\mu| \) coincides with the sequence of moments of the Gamma
distribution with
    shape parameter \( 3r \): if its density is \( f(x) = x^{3r} e^{-x} /
    \Gamma(3r) \), then \( k \)-th moment is calculated as
    \begin{equation}
        \int_0^\infty
        \dfrac{x^{3r + k - 1}}{\Gamma(3r)} e^{-x}
        dx
        =
        \dfrac{\Gamma(3r+k)}{\Gamma(3r)!}.
    \end{equation}
    By checking Carleman's condition for Stieltjes moment problem on \( (0,
    +\infty) \), we conclude that this distribution is uniquely defined by its
    moments, which finishes the proof.
\end{proof}

\begin{corollary}
    Since in the subcritical phase a random 2-CNF has a contradictory component
    of excess \( 1 \) with probability \( \frac{1}{16 |\mu|^{3}} \),
    and components of
    higher excess with a negligible probability, the distribution of the number
    of contradictory variables can be approximated by a mixture of
    a deterministic value \( 0 \) with probability
    \( \mathbb P(V = 0) = 1 - \frac{1}{16 |\mu|^{3}} \)
    and
    of Gamma distribution with parameter \( 3 \) and scale
    \( n^{1/3} |\mu|^{-1} \) with probability \( \frac{1}{16 |\mu|^{3}} \).
\end{corollary}

\subsection{Structure of the spine}
In the paper~\cite{bollobas2001scaling} it is proven that the expected size of
the spine in the subcritical phase, \ie for \( m = n(1 + \mu n^{-1/3}) \) when
\( \mu \to -\infty \) with \( n \), is asymptotically \( \frac12 |\mu|^{-2}
n^{2/3} \).
As proven in~\cref{lemma:minimal:spinal:path}, for every literal \( y \) from
the spine of a formula \( F \in \mathcal F(n,m) \) there exists a minimal spinal
path of the form \( y \rightsquigarrow x \rightsquigarrow \n x \), such that all
the internal nodes of the paths \( y \rightsquigarrow x \) and \( x
\rightsquigarrow \n x\) are pairwise strictly distinct. We show that for almost
all literals \( y \) from the spine such a minimal spinal path is unique.

\begin{theorem}
    \label{theorem:spine}
    Consider random formulae \( F \in \mathcal F(n,m) \) in the subcritical phase,
    \( m = n(1 + \mu n^{-1/3}) \), \( \mu \to -\infty \), \( |\mu| \leq n^{1/12}
    \).
    The expected number of spine variables \( y \) that have exactly \( k \)
    unique paths from \( y \) to \( \n y \) is asymptotically equal to \( C_k
    n^{2/3} |\mu|^{-2-3k} \), where \( C_k \) is some algorithmically computable
    constant.
    In particular, the expected proportion of spine variables \( y \) having a
    unique path \( y \rightsquigarrow \n y \) is
    \( 1 - O(|\mu|^{-3}) \).
\end{theorem}

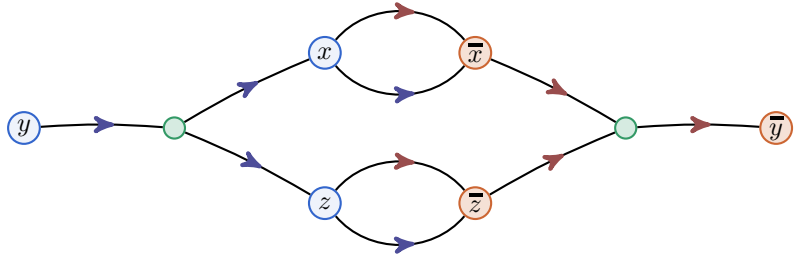
\begin{figure}[hbt]
    \centering
\begin{tikzpicture}[>=stealth',thick, node distance=1.0cm]
\draw
node[arnBleuPetit, minimum height = 1.2em,  thick](x)
                    at (0,1) {\( x \)}
node[arnRougePetit, minimum height = 1.2em, thick](x')
                    at (2,1) {\( \fn x \)}
node[arnBleuPetit, minimum height = 1.2em,  thick](z)
                    at (0,-1) {\( z \)}
node[arnRougePetit, minimum height = 1.2em, thick](z')
                    at (2,-1) {\( \fn z \)}
node[arnVertPetit](w)
                    at (-2,0) {}
node[arnVertPetit](w')
                    at (4,0) {}
node[arnBleuPetit, minimum height = 1.2em,  thick](y)
                    at (-4,0){\( y \)}
node[arnRougePetit, minimum height = 1.2em, thick](y')
                    at (6,0) {\( \fn y \)}
;
\fromto{x}{x'}     [ultra thick,blackred] [thick]
                    [-50][-5][0.6];
\fromto{x}{x'}     [ultra thick,blackblue][thick]
                    [50][5][0.6];
\fromto{z}{z'}     [ultra thick,blackred] [thick]
                    [-50][-5][0.6];
\fromto{z}{z'}     [ultra thick,blackblue][thick]
                    [50][5][0.6];
\fromto{w'}{y'}     [ultra thick,blackred][thick]
                    [-4][-1][0.6];
\fromto{y}{w}     [ultra thick,blackblue] [thick]
                    [-4][-1][0.6];
\fromto{x'}{w'}     [ultra thick,blackred][thick]
                    [-4][-1][0.6];
\fromto{w}{x}     [ultra thick,blackblue] [thick]
                    [-4][-1][0.6];
\fromto{z'}{w'}     [ultra thick,blackred][thick]
                    [-4][-1][0.6];
\fromto{w}{z}     [ultra thick,blackblue] [thick]
                    [-4][-1][0.6];
\end{tikzpicture}
\caption{\label{fig:spine:further}
    One possible configuration of two distinct paths \( y \rightsquigarrow \n y
    \).}
\end{figure}

\begin{proof}
    Let a random variable \( P_n \) denote the number of spine literals \( y
    \) of a random formula \( F \in \mathcal F(n,m) \),
    each variable counted with a multiplicity of the number of paths \( y
    \rightsquigarrow \n y \). We also define \( P_n^{(2)} \) which counts the
    number of spine literals \( y \) counted with the multiplicities of the
    pairs of distinct paths \( y \rightsquigarrow \n y \), and similarly
    \( P_n^{(\ell)} \) for \( \ell \)-tuples of distinct paths.

    The cardinality of the spine \( \mathcal S(F) \),
    \ie the number of literals \( y \) for which there exists at least one such
    path, can be then counted using the variant of the inclusion-exclusion
    approach:
    \begin{equation}
        \mathbb E \mathcal S(F) = \mathbb E P_n
        - \mathbb E P_n^{(2)}
        + \dfrac{1}{2!}
        \mathbb E P_n^{(3)}
        - \dfrac{1}{3!}
        \mathbb E P_n^{(4)}
        + \cdots
        .
    \end{equation}
    In order to prove the theorem, we show that
    \( \mathbb E P_n = \Theta(n^{2/3} \mu^{-2}) \),
    \( \mathbb E P_n^{(2)} = \Theta(n^{2/3} \mu^{-5}) \),
    and, generically,
    \( \mathbb E P_n^{(k)} = \Theta(n^{2/3} \mu^{-3k + 1}) \).

    The expected value of minimal spinal paths can be again counted using
    sum-representations. By distinguishing a pattern \( y \rightsquigarrow x
    \rightsquigarrow \n x \) and counting such graphs with weight \( 2^\ell \)
    where \( \ell \) is the length of the pattern, we are counting the minimal
    spinal paths with multiplicity \( 2 \), as there are \( 2 \) distinct paths
    \( x \rightsquigarrow \n x \).

    Knowing that exclusion of the conflicting edges gives the same result as
    counting the proportions of graphs with a distinguished pattern in simple
    digraphs (see~\cref{subsection:excluding:non:sum:representations}), we pass
    directly to the counting in simple digraphs. The corresponding EGF for
    simple digraphs \( D \in \mathcal D(2n, m) \) with distinguished pattern
    which forms a weakly connected component which is a tree, is then (taking
    into account the compensating factor \( 1/2 \) arising from the multiplicity
    mentioned above)
    \[
        \dfrac{1}{2}
        \Lambda \left(
            \dfrac{U_\to(z)^{2n-m-1}}
            {(2n-m-1)!} e^{V_\to(z)}
            {4 T_\to(z)^3  z^{-2}}
            { (1 - 2 T_\to(z))^2}
        \right)
    \]
    and therefore, by applying~\cref{lemma:sp:digraphs}, we obtain the expected
    value of the dominant term
    \[
        \mathbb E P_n
        \sim
        \dfrac{1}{2} \cdot \dfrac{1}{2n} \cdot 4 \cdot n \cdot 4 \cdot \dfrac{1}{8}
        n^{2/3} \mu^{-2}
        =
        \dfrac{1}{2}
        n^{2/3} \mu^{-2}
        .
    \]
    The above expectation also includes the cases when the excess of the weakly
    connected component containing the pattern \( y \rightsquigarrow x
    \rightsquigarrow \n x \) is greater than \( -1 \), but such cases give
    contribution at most \( O(n^{2/3} \mu^{-5}) \) and therefore are negligible.

    Similarly, by constructing all possible cases of having two distinct paths
    between \( y \) and \( \n y \) in the implication digraph, we note that the
    dominant contributions come only from the case when vertices of degree \(
    3 \) are inserted into the core of the pattern, and therefore the pattern is
    almost cubic (except for the nodes \( y, x \) and \( \n x \), and the pairs
    of complementary literals, where each such pair has a summary degree \( 3
    \)). It is easy to show that such components contribute a factor
    \( O(n^{2/3} \mu^{-5}) \) (one example of such configuration is given in~\cref{fig:spine:further}).
    Considering spine literals of further complexity
    gives the next orders of asymptotics. This observation finishes the proof.
\end{proof}

\begin{corollary}
    \label{corollary:spine}
    By removing on average a \( \Theta(|\mu|^{-3}) \) proportion of the spine
    literals, the spine breaks down into non-intersecting tree-like components,
    viz.~\cref{fig:spine:tree}. For each component there exists a distinct
    literal \( x \) such that every literal \( y \) from this component has a
    unique path to \( x \), and the path \( x \rightsquigarrow \n x\) is unique
    and strictly distinct.
\end{corollary}

\section{Conclusions and open problems}
\label{section:open:problems}

While the analysis of random graphs and the curve of their phase transition have
been described in details in~\cite{janson1993birth}, there is a certain obstacle
which doesn't immediately allow to go inside the critical phase
of 2-SAT phase transition using the inclusion-exclusion approach.
In~\cite{collet2018threshold}, a notion of a
\emph{patchwork} has been specifically designed for a very similar reason.
However, their approach also doesn't suggest any explicit way for obtaining such
patchworks. These considerations give rise to the following questions.

\begin{problem}
    \label{problem:enumeration:of:cscc}
    What is the number of cubic strongly connected contradictory multigraph
    components with given excess? How many minimal contradictory (cubic)
    multigraphs implication multigraphs are there?
\end{problem}

A variation of the above question, directly related to the inclusion-exclusion
method, leads to the following question.

\begin{problem}
    Using a version of~\cite[Lemma 3]{janson1993birth} when \( |\mu| = \Theta(1)
    \), is it possible to express the probability of
    satisfiability as a \emph{converging} sum of Airy functions \( A(y, \mu) \)?
\end{problem}

Some of the simulations suggested that the introduced notion of the excess
correctly captures the discrete nature of the phase transition. According to the
results, the distribution of the excess is discrete for finite \( \mu \) and
doesn't depend on \( n \) in the limit.

\begin{conjecture}
    \label{conjecture:discrete:excess}
Let \( \xi \) denote the total excess of the strongly connected contradictory
component in a random implication digraph corresponding to a formula \( F \)
chosen uniformly at random from \( \mathcal F(n, m) \).
If \( \mu \) is constant and \( m = n (1 + \mu n^{-1/3}) \), the distribution of
\( \xi \) is discrete. Moreover, when \( \mu \) is constant, the number of
strongly connected components of the contradictory component digraph follows a
discrete limiting law; when \( \mu \to +\infty \), the contradictory component
is strongly connected with high probability.
\end{conjecture}

When \( \mu \to +\infty \), the expected value of the excess of the complex
component in simple graphs is known and is \( \frac{2}{3} \mu^3 \). Some
simulations suggest that exactly the same asymptotics may hold
when \( \mu \to \infty \) for the expectated excess of the contradictory
component.

\begin{problem}
    When \( m = n (1 + \mu n^{-1/3}) \) and \( \mu \to +\infty \), what is the
    expected excess of the contradictory component of a random 2-CNF formula?
\end{problem}

One of the motivations to study the phase transition in \( 2 \)-SAT is its
similarity to the phase transition of the appearance of the strongly connected
component in directed graphs. The papers~\cite{luczak2009critical,pittel2017birth}
seem to come the closest to resolving the question, but is it possible to give
the exact description?

\begin{problem}
    It is possible to describe a ``giant strong component'' of a critical
    directed graph having \( n \) vertices and \( m = n(1 + \mu n^{-1/3}) \)
    oriented edges in the terms of cubic components and their excesses?
    Is it then possible to express the probability of having a strong component of
    excess \( r \) in terms of Airy function depending on \( r \)?
\end{problem}

One of the applications of analysis with the help of generating functions, is
the study of \( 2 \)-SAT with a given set of degree constraints, similarly
to~\cite{de2016graphs} and~\cite{dovgal2018shifting}. It is clear that the
same analysis can be done for the case of formulae with literal set degree
constraints, however the present analysis is done only for the subcritical
phase.

\begin{conjecture}
    The point \( r \) of the phase transition in the 2-SAT model with literal
    degree constraints from a set \( \Delta = \{ d_1, d_2, \ldots \} \)
    (possibly weighted) can be computed from the system of equations
    \begin{align}
        z \dfrac{\omega''(z)}{\omega'(z)} = 1,
        \quad
        z \dfrac{\omega'(z)}{\omega(z)} = r,
    \end{align}
    where \( \omega(z) := \sum_{d \in \Delta} \dfrac{z^d}{d!} \) and \( \Delta
    \) satisfies the condition \( 1 \in \Delta \).
\end{conjecture}

\subsection*{Acknowledgements.} The author is indebted to \'{E}lie de
Panafieu and Vlady Ravelomanana for uncountable very inspiring discussions and
for their support and encouragement.

{\footnotesize
\bibliography{birth-contradictory}
}

\appendix

\section{Complete asymptotic expansion for the saddle point lemma}
\label{section:full:asymptotic:expansion}
Let us show how a complete asymptotic expansion in powers of \( \mu^{-3} \) can
be obtained in~\cref{lemma:sp:integration}. In~\cite{janson1993birth}, the
following formulation is given.
\begin{lemma}[{\cite[Lemma 3]{janson1993birth}}]
    \label{lemma:original:contour:integration}
    If \( m = \tfrac12 (1 + \mu n^{-1/3}) \) and \( y \) is any real constant,
    \( \mathcal{MG}(n,m) \) is the set of all labelled multigraphs with \( n \)
    vertices and \( m \) edges, then, as \( \mu \to -\infty \) with \( n \)
    while remaining \( |\mu| \leq n^{1/12} \),
    \[
        \dfrac{n!}{|\mathcal{MG}(n,m)|}
        [z^n] \dfrac{U(z)^{n-m}}{(n-m)! (1 - T(z))^y}
        = \sqrt{2 \pi} A(y, \mu) n^{y/3 - 1/6}
        \Big(
            1 + O(\mu^4 n^{-1/3})
        \Big)
        ,
    \]
    where
    \[
        A(y, \mu) =
        \dfrac{1}{\sqrt{2 \pi} |\mu|^{y - 1/2}}
        \left(
            1 - \dfrac{3 y^2 + 3 y - 1}{6 |\mu|^3}
            + O(\mu^{-6})
        \right)
        .
    \]
\end{lemma}

\noindent
\textbf{Complete asymptotic expansion.}
In the proof, it is mentioned that in principle, it is possible to obtain a
complete asymptotic series of \( |\mu|^{-3} \). Let us describe the procedure
that can be used to compute these coefficients.

Let \( \alpha = - \mu \). As shown in the proof of~\cite[Lemma
3]{janson1993birth}, the function \( A(y, \mu) \) can be represented in the form
\[
    A(y,\mu) =
    \dfrac{1}{2 \pi \alpha^{y - 1/2}}
    \int_{-\infty}^{\infty}
    \left(
        1 + \dfrac{it}{\alpha^{3/2}}
    \right)^{1-y}
    e^{-t^2/2 - it^3 / (3 \alpha^{3/2})}
    dt
    .
\]
In order to express \( A(y,\mu) \) in the form of a complete asymptotic
expansion, we introduce \( \beta := i\alpha^{-3/2} \) and obtain:
\[
    A(y,\mu)
    =
    \dfrac{1}{\sqrt{2 \pi} |\mu|^{y - 1/2}}
    \int_{-\infty}^{+\infty}
    \left(
        1 + \beta t
    \right)^{1-y}
    e^{ - \beta t^3 / 3 }
    dt
    \sim
    \dfrac{1}{2 \pi |\mu|^{y - 1/2}}
    \sum_{r \geq 0} c_r(y) \beta^r
    ,
\]
where \( (c_r(y))_{r=0}^\infty \) are polynomials in \( y \).
The coefficient \( [\beta^k]
    \left(
        1 + \beta t
    \right)^{1-y}
    e^{ - \beta t^3 / 3 }
\) can be expressed as the convolution of two generating functions
\[
    [\beta^r]
    \left(
        1 + \beta t
    \right)^{1-y}
    e^{ - \beta t^3 / 3 }
    =
    \sum_{k = 0}^r
    t^k
    { 1 - y \choose k}
    \dfrac{(- t^3 / 3)^{r-k}}{(r-k)!}
    =
    \sum_{k = 0}^r
    \dfrac{(-1/3)^{r-k}}{(r-k)!}
    t^{3r - 2k}
    { 1 - y \choose k}
\]
A formal term by term integration (the series is most likely not convergent, but
the expansion could be extended up to \( r \)th term for any finite \( r \))
yields for even \( r \)
\[
    c_{2r}(y) =
    \sqrt{2 \pi}
    \sum_{k=0}^{2r}
    \dfrac{(-1/3)^{2r-k}}{(2r - k)!}
    {1 - y \choose k}
    \int_{-\infty}^{\infty}
    e^{-t^2/2}
    t^{6r - 2k}
    dt
    =
    \sum_{k=0}^{2r}
    \dfrac{(-1/3)^{2r-k}}{(2r - k)!}
    {1 - y \choose k}
    (6r - 2k - 1)!!
    ,
\]
where the double factorial notation is used \( (2n-1)!! := 1 \cdot 3 \cdot
\cdots \cdot (2n-1) \); for odd \( r \) the principal value of the integral
equals zero, and so, \( c_{2r + 1}(y) = 0 \).

As an example, the first nontrivial term \( c_2(y) \) can be computed as
\[
    c_{2}(y) = \dfrac{1/9}{2!} \cdot 5!!
    - \dfrac{1/3}{1!} \cdot (1 - y) \cdot 3!!
    + \dfrac{(1 - y)(-y)}{2} \cdot 1!!
    = \dfrac{3 y^2 + 3y - 1}{6}
\]
and a factor \( -1 \) in the terms \( c_{4r + 2}(y) \) in the expansion of \(
A(y,\mu) \) appears because a multiple \( i^{4r + 2} = -1 \)
should be extracted from \( \beta = i \alpha^{-3/2} \).

\end{document}